\documentclass[b5paper,reqno]{udt9}
\usepackage{amsmath,amssymb,amsfonts,amsthm,latexsym,amstext,amscd,epstopdf}
\usepackage{epsfig,graphics,graphicx,color}
\usepackage{enumerate}
\usepackage{hyperref}

 \usepackage{geometry}
\usepackage{mathrsfs}

\usepackage{vmargin}

\newtheorem{theorem}{Theorem}[section]

\newtheorem{proposition}[theorem]{Proposition}
\newtheorem{corollary}[theorem]{Corollary}
\theoremstyle{definition}
\newtheorem{definition}[theorem]{Definition}
\newtheorem{remark}[theorem]{Remark}

\newcommand{\scp}[1]{\langle#1\rangle}

\makeatletter \thm@headfont{\bfseries\scshape} \makeatother
\allowdisplaybreaks

\issueinfo{11}{2016}{2}%
\DOI{00.0000/UDT--2016--0000}
\pagespan{0}{0}%
\received{April 17, 2015}%
\accepted{March 7, 2016}%
\begin{document}
\title[Integral powers of numbers in small intervals
modulo 1]
      {Integral powers of numbers in small intervals
modulo 1: The cardinality gap phenomenon}
\author[Johannes Schleischitz]
       {Johannes Schleischitz}

\address{{\bf{Augasse 2-6, 1090 Vienna}}\\
             }
\email{johannes.schleischitz@boku.ac.at}

\keywords{distribution modulo $1$, distribution of powers, Pisot number, Salem numbers, Mahler's 3/2-problem}

\subjclass{11B05, 11J71, 11K16, 11K31}

 \thanks{supported by FWF grant P24828}

\begin{abstract}
This paper deals with the distribution of $\alpha \zeta^{n} \bmod 1$, where
$\alpha\neq 0,\zeta>1$ are fixed real numbers and $n$ runs through the positive integers.
Denote by $\Vert.\Vert$ the distance to the nearest integer.
We investigate the case of $\alpha\zeta^{n}$ all lying in prescribed small intervals modulo $1$
for all large $n$, with focus on the case $\Vert\alpha \zeta^{n}\Vert \leq \epsilon$ for small $\epsilon>0$.
We are particularly interested in what we call cardinality gap phenomena.
For example for fixed $\zeta>1$ and small $\epsilon>0$
there are at most countably many values of $\alpha$ such that $\Vert\alpha \zeta^{n}\Vert \leq \epsilon$
for all large $n$, whereas larger $\epsilon$ induces an uncountable set.
We investigate the value of $\epsilon$ at which the gap occurs.
We will pay particular attention to the case of algebraic and, more specific,
rational $\zeta>1$. Results concerning Pisot and Salem numbers
such as some contribution to Mahler's $3/2$-problem
are implicitly deduced. We study similar questions for fixed $\alpha\neq 0$ as well.

 \end{abstract}%

\vskip0.5cm%

\maketitle

\begin{centerline}%
{\it Communicated by Michael Drmota}
\end{centerline}
\vskip0.5cm%

\section{Notation and known results} \label{ggg}

This paper aims to study the distribution of $\alpha\zeta^{n}$ mod $1$ for real numbers $\alpha\neq 0,\zeta>1$.
We start with some definitions concerning representations of numbers modulo $1$.

\begin{definition}
For $x\in{\mathbb{R}}$ denote by $\lfloor x\rfloor\in{\mathbb{Z}}$ the largest integer smaller or equal $x$,
and $\lceil x\rceil\in{\mathbb{Z}}$ the smallest integer greater or equal $x$. Let further $\{x\}\in{[0,1)}$ be the fractional part of $x$, i.e. $\{x\}=x-\lfloor x\rfloor$.
Furthermore denote by $\scp{x}\in{\mathbb{Z}}$ the integer closest to $x$,
where clearly $\scp{x}\in{\{\lfloor x\rfloor, \lceil x\rceil\}}$.
(In the special case $\{x\}=1/2$ let $\scp{x}:= \lfloor x\rfloor$, however it
will not be of much interest in the sequel.)
Finally, denote by $\Vert x\Vert:=\vert x-\scp{x}\vert\in{[0,1/2]}$ the distance
from $x$ to the nearest integer.
\end{definition}

\begin{definition}
For set $A$ denote by $\vert A\vert$ the cardinality of $A$.
\end{definition}

The following theorem comprises two important metric uniform distribution results.
One is due to Weyl~\cite{25} and the other due to Koksma~\cite{26}.

\begin{theorem}[Weyl, Koksma] \label{kok}
For any fixed real $\alpha\neq 0$, for almost all $\zeta>1$ the sequence $\{\alpha\zeta^{n}\}$
is uniformly distributed modulo $1$. For any fixed $\zeta>1$, for almost all real $\alpha$
the sequence $\{\alpha\zeta^{n}\}$ is uniformly distributed modulo $1$.
\end{theorem}

We want to investigate the set of $\alpha,\zeta$ with the property that $\alpha\zeta^{n}$
is close to integers for all $n\geq n_{0}$.
Theorem~\ref{kok} shows that this is a highly non-generic set of $(\alpha,\zeta)\subseteq{\mathbb{R}^{2}}$.
Examples of numbers in the exceptional set of Theorem~\ref{kok}
are given for $\zeta$ a Pisot number or a Salem numbers and suitable $\alpha$.
Pisot numbers are defined as real algebraic integers greater
than $1$ whose proper conjugates all lie strictly inside the unit circle in $\mathbb{C}$, whereas Salem
numbers are real algebraic integers greater than $1$
having all proper conjugates in the closed unit circle with at least one on the torus.
Some basic facts on Pisot and Salem numbers that can be found in~\cite[Chapter~5]{27}
are summarized in the following theorem.

\begin{theorem}[Pisot]  \label{pisot}
Let $\zeta$ be a Pisot number. Then $\lim_{n\to\infty}\Vert\zeta^{n}\Vert=0$.
This property characterizes Pisot numbers among all real algebraic numbers greater than one.
Even the two following stronger assertions holds:
if either $\Vert\alpha\zeta^{n}\Vert$ tends to $0$ for a real algebraic number $\zeta>1$ and some $\alpha\neq 0$,
or if $\sum_{n=1}^{\infty}\Vert \alpha\zeta^{n}\Vert^{2}<\infty$ for arbitrary $\zeta>1$ and some $\alpha\neq 0$,
then $\zeta$ is a Pisot number and $\alpha$ belongs to the number field $\mathbb{Q}(\zeta)$.

Now let $\zeta$ be a Salem number. Then the sequence $\{ \zeta^{n} \}$ is dense in $(0,1)$ but not
uniformly distributed. For any $\nu\in{(0,1/2)}$, there exists $\alpha$ such that
$\Vert \alpha\zeta^{n}\Vert<\nu$ for $n\geq n_{0}$ and the sequence $(\alpha\zeta^{n})_{n\geq 1}$
is dense modulo $1$ in the symmetric interval of length $2\nu$ and center $0$.
\end{theorem}

The convergence results for Pisot numbers can be generalized and refined in some ways.
However for our purposes the above is sufficient, and we just refer to~\cite{27}.
It is an open question if any real transcendental number $\zeta$ has the property that
for some $\alpha\neq 0$ the expression $\Vert\alpha\zeta^{n}\Vert$ tends to $0$ as $n\to\infty$.
This motivates to look at $\alpha,\zeta$ with $\alpha\zeta^{n}$ close to integers,
in particular $\Vert \alpha\zeta^{n}\Vert \leq \epsilon$ for some $\epsilon>0$ and all large $n$.
We quote some results connected to this question,
which can be found in~\cite[Chapter~5]{27} unless quoted otherwise.

\begin{theorem} \label{bertin}
The set of pairs $(\alpha,\zeta)\in{\mathbb{R}^{2}}$ with $\alpha>0, \zeta>1$, such that
\[
\sup_{n\geq n_{0}} \Vert \alpha\zeta^{n}\Vert \leq \frac{1}{2(1+\zeta)^{2}}
\]
holds for an integer $n_{0}$, is countable.
\end{theorem}

Theorem~\ref{pisot} and Theorem~\ref{bertin} imply that the set of
pairs $(\alpha,\zeta)\in{\mathbb{R}\times (1,\infty)}$ with the property
$\lim_{n\to\infty}\Vert\alpha\zeta^{n}\Vert=0$ is countable infinite.

\begin{theorem}
Let $\zeta>1$ be a real number. Suppose there exists $\alpha\geq 1$ such that
\[
\Vert \alpha\zeta^{n}\Vert \leq \frac{1}{2e\zeta(\zeta+1)(\log \alpha+1)}, \qquad n\geq 1.
\]
Then $\zeta$ is either a Pisot number or a Salem number and $\alpha\in{\mathbb{Q}(\zeta)}$.
\end{theorem}

\begin{theorem}  \label{fischersfritz}
Let $\zeta>1$ be a real number. Suppose there exists $\alpha\geq 1$ such that
\[
\Vert \alpha\zeta^{n}\Vert \leq \frac{1}{e(\zeta+1)^{2}(\sqrt{\log \alpha}+2)}, \qquad n\geq 1.
\]
Then $\zeta$ is either a Pisot number or a Salem number and $\alpha\in{\mathbb{Q}(\zeta)}$.
\end{theorem}

Reverse examples are due to Boyd~\cite{boyd}.

\begin{theorem}[Boyd] \label{boyd}
There are arbitrarily large transcendental $\zeta>3$ and $\alpha$ {\upshape(}depending on $\zeta${\upshape)}
arbitrarily close to $2$, such that
\[
\Vert \alpha\zeta^{n}\Vert \leq \frac{1}{(\zeta-1)(\zeta-3)}, \qquad n\geq 0.
\]
There exist real transcendental $\zeta>1$ such that for some
$\alpha\geq 1$ {\upshape(}depending on $\zeta${\upshape)}
\[
\Vert \alpha\zeta^{n}\Vert \leq \frac{5}{e\zeta(\zeta+1)(\log \alpha+1)}, \qquad n\geq 1.
\]
\end{theorem}
Another result for the special case $\alpha=1$ we want to quote is~\cite[Corollary~5]{dubickas}.

\begin{theorem} [Dubickas]  \label{adub}
Let $(r_{n})_{n\geq 1}$ be a sequence of real numbers. Then, for any $\epsilon>0$, there is
$\zeta>1$ such that $\Vert \zeta^{n}-r_{n}\Vert<\epsilon$ for each $n\geq 1$.
\end{theorem}

Restricting to large $n$, we will refine Theorem~\ref{adub} in Section~\ref{kkk}.
Finally we state~\cite[Theorem~3]{mosh}, which also refines Theorem~\ref{adub}.

\begin{theorem}[Bugeaud, Moshchevitin]  \label{mosh}
Let $\alpha$ be a positive real number. Let $\epsilon<1$ be a positive real number.
Let $(a_{n})_{n\geq 1}$ be a sequence of real numbers satisfying $0\leq a_{n}< 1-\epsilon$
for all $n\geq 1$. The set of real numbers $\zeta$ such that $a_{n}\leq \{\alpha \zeta^{n}\}\leq a_{n}+\epsilon$
for every $n\geq 1$ has full Hausdorff dimension.
\end{theorem}

Observe that Theorem~\ref{mosh} is somehow reverse to Theorem~\ref{kok}.
The analogue of Theorem~\ref{mosh} with the roles of $\alpha$ and $\zeta$ exchanged fails
heavily. We will see in Section~\ref{sektion3} (resp. Section~\ref{sektions})
that generic algebraic (resp. rational) $\zeta>1$ provide counterexamples.
We will at some places consider a more general situation, in which the following
theorem due to Pollington~\cite{pollington} suits.

\begin{theorem}[Pollington] \label{pollington}
Let $(t_{n})_{n\geq 1}$ be a sequence of positive numbers such that
\[
q_{n}:=\frac{t_{n+1}}{t_{n}}\geq \delta>1, \qquad n\geq 1.
\]
Further let $s_{0}\in{(0,1)}$. Then there exists a real number $\beta=\beta(\delta,s_{0})>0$
and a set $T$ of Hausdorff dimension at least $s_{0}$ such that if $\alpha\in{T}$ then
\[
\{t_{n}\alpha\}\in{[\beta,1-\beta]}, \qquad n\geq 1.
\]
Concretely $\beta$ may be chosen $(1/2)(r+1)^{-1}\delta^{-4r}$, where
$r$ is sufficiently large that $\delta^{r}-(r+2)>\delta^{rs_{0}}$.
In particular, the set of $\alpha$ such that $\{t_{k}\alpha\}$ is not dense in $[0,1)$
has full dimension.
\end{theorem}

The explicit bounds in dependence of $\beta$ are not explicitly stated in
the formulation of the central theorem of~\cite[page~511]{pollington}, but were
in fact established in the paper, see the formulas (3),(4) and (4a) in~\cite{pollington}.

\section{Outline of selected results} \label{outline}

We outline the most important results which we will establish.
However, we point out that in the course their proofs, several other
results will be derived that are of some interest on their own and not covered in the current section.
In particular the results concerning
the case of fixed $\alpha$ in Section~\ref{fixedalpha} and the first
part of Section~\ref{epsfix} are self-contained and not part of this overview.

Our first selected result deals with the root distribution of polynomials with integral coefficients.
It arises as a corollary of our study of the
sequences $(\alpha\zeta^{n})_{n\geq 1}$, combined with a result due to Dubickas.
As usual let $L(P):=\sum_{i=0}^{m} \vert a_{i}\vert$
for a polynomial $P(X)=a_{0}+a_{1}X+\cdots+a_{m}X^{m}$,
and $L(\zeta)=L(P)$ for an algebraic number $\zeta$ where $P\in{\mathbb{Z}[X]}$ is the minimal
polynomial of $\zeta$ in lowest terms.
\begin{theorem} \label{hot}
Assume real algebraic $\zeta$ satisfies $2(\zeta-1)>L(\zeta)$.
Then $\zeta$ is a Pisot number. In other words, if $P\in{\mathbb{Z}[X]}$
has a real root larger than $L(P)/2+1$, all the other roots of $P$ lie inside the unit circle.
\end{theorem}
We compare Theorem~\ref{hot}
with the well-known bounds
\begin{align}
\max_{1\leq j\leq m}\vert \zeta_{j}\vert&\leq 1+\frac{\max_{i\neq m}\vert a_{i}\vert}{\vert a_{m}\vert}\leq 1+\frac{H(P)}{\vert a_{m}\vert}
\leq 1+\frac{L(P)}{\vert a_{m}\vert}  \label{eq:hoehet} \\
M(P)&:=\vert a_{m}\vert\prod_{j=1}^{m} \max\{1,\vert \zeta_{j}\vert\}\leq \Vert P\Vert_{2}:=
\sqrt{\sum_{i=0}^{m} \vert a_{i}\vert^{2}}\leq L(P),  \label{eq:tiefet}
\end{align}
for arbitrary $P(X)=a_{m}X^{m}+a_{1}X+\cdots+a_{0}\in{\mathbb{C}[X]}$, see~\cite{mignotte}.
Here $\zeta_{j}$ are the roots of $P$
and $H(P)=\max_{0\leq j\leq m} \vert a_{i}\vert$. In view of \eqref{eq:hoehet},
the existence of a root as in the last claim of Theorem~\ref{hot} requires that $P$ is monic.
In this case combination of \eqref{eq:tiefet} and the assumption of Theorem~\ref{hot}
yield that the remaining roots have modulus less than $L(P)/(L(P)/2+1)<2$, a weaker conclusion
than Theorem~\ref{hot}. It is easy to construct non-trivial $P\in{\mathbb{Z}[X]}$ with the inferred
bound arbitrarily close to $2$. Relations between $l(P), L(P), M(P)$ have
been studied by Dubickas~\cite{dubbull} and Schinzel~\cite{schinzel},~\cite{schinzel2},~\cite{schinzel3}.

For Pisot numbers the inequality $2(\zeta-1)>L(\zeta)$ can be satisfied.
Take for instance $\zeta=\zeta_{m,b}$ the Pisot root of $P_{m,b}(X)=X^{m}-bX^{m-1}-1$ for integers $m\geq 2, b\geq 4$.
Indeed $P_{m,b}$ has a root in $(b,b+1)$ by intermediate value theorem
and $L(P_{m,b})=b+2$, so $b\geq 4$ is certainly sufficient
for $2(\zeta_{m,b}-1)>L(\zeta_{m,b})$ and Rouchee's~Theorem implies that these polynomials are indeed Pisot polynomials
(Theorem~\ref{hot} also implies $P_{m,b}$ is a Pisot polynomial).
In fact the expression $L(\zeta_{m,b})/(\zeta_{m,b}-1)$ tends to $1$ as $b\to\infty$.
On the other hand, a Pisot number need not satisfy $\zeta>L(\zeta)/2+1$,
for instance take $\zeta=\zeta_{m,b}$ with $m=2$ and $b\in{\{1,2,3\}}$. Thus Theorem~\ref{hot}
only yields a sufficient condition for an algebraic number to be a Pisot number.

For the other results we need to introduce some notation.

\begin{definition}  \label{wdeff}
For real numbers $\zeta>1,\epsilon>0$, let $\varpi_{\epsilon,\zeta}$ be
the set of all real $\alpha\neq 0$ such that
$\Vert \alpha\zeta^{n}\Vert\leq \epsilon$ for all $n\geq n_{0}(\alpha,\zeta,\epsilon)$.
\end{definition}

Obviously $\varpi_{\epsilon_{0},\zeta}\subseteq \varpi_{\epsilon_{1},\zeta}$
for $\epsilon_{0}<\epsilon_{1}$ and any $\zeta$. Note also that
$\varpi_{\epsilon,\zeta}\neq \emptyset$ for all $\epsilon>0$ is a necessary condition
for $\lim_{n\to\infty}\Vert\alpha\zeta^{n}\Vert=0$.
In fact, $\bigcap_{\epsilon>0} \varpi_{\epsilon,\zeta}$ is the set of values $\alpha$
such that $\lim_{n\to\infty} \Vert \alpha\zeta^{n}\Vert=0$ for a fixed $\zeta$.
The sets $\varpi_{\epsilon,\zeta}$ are obviously closed under any map $\tau_{k,\zeta}:\alpha\mapsto \alpha\zeta^{k}$
for $k$ a positive integer. We investigate the cardinality of these sets.
More precisely, our focus is on understanding the derived quantities
\begin{align*}
\tilde{\epsilon}_{1}&=\tilde{\epsilon}_{1}(\zeta):=
\sup \left\{\epsilon>0: \varpi_{\epsilon,\zeta}=\emptyset \right\}
=\inf \left\{ \epsilon>0: \left\vert \varpi_{\epsilon,\zeta} \right\vert\geq \vert \mathbb{Z}\vert \right\}, \\
\tilde{\epsilon}_{2}&=\tilde{\epsilon}_{2}(\zeta):=
\sup \left\{ \epsilon>0: \left\vert \varpi_{\epsilon,\zeta} \right\vert\leq \vert \mathbb{Z}\vert \right\}
=\inf \left\{ \epsilon>0: \left\vert \varpi_{\epsilon,\zeta} \right\vert>\vert \mathbb{Z}\vert \right\}.
\end{align*}
An equivalent definition of $\tilde{\epsilon}_{1}$ is given by
\[
\tilde{\epsilon}_{1}(\zeta)= \inf_{\alpha\neq 0} \; \limsup_{n\to\infty} \Vert \alpha\zeta^{n}\Vert.
\]
Obviously $0\leq \tilde{\epsilon}_{1}(\zeta)\leq \tilde{\epsilon}_{2}(\zeta)\leq 1/2$ for all real $\zeta$.
We will establish the better bounds given in the following theorem.

\begin{theorem} \label{epsilont}
For any $\zeta>1$ we have
\[
0\leq \tilde{\epsilon}_{1}(\zeta)\leq \min\left\{\frac{1}{2},\frac{1}{2(\zeta-1)}\right\},
\qquad \frac{1}{2(\zeta+1)}\leq \tilde{\epsilon}_{2}(\zeta)\leq \min\left\{\frac{1}{2},\frac{1}{\zeta-1}\right\}.
\]
\end{theorem}

\begin{remark}
For rather small values of $\zeta$ the upper bound $1/2$ in Theorem~\ref{epsilont}
can be slightly reduced, provided a slight modification of Pollington's result holds.
Assume Theorem~\ref{pollington} with the same effective bound for $\beta$
is valid if the fractional parts $\{t_{k}\alpha\}$ avoid the
interval $(1/2-\beta,1/2+\beta)$ instead of the open intervals of the same length $2\beta$ around integers
as in the theorem.
Looking at the proof of Theorem~\ref{pollington} in~\cite{pollington} this shift invariance seems very reasonable.
Put $t_{n}=\zeta^{n}$ and observe we may let $s_{0}>0$ be arbitrarily small and still obtain
uncountably many elements $\alpha$ with the desired property. Thus with $r=r(\zeta)$ the smallest
positive integer with $\zeta^{r}>r+3$, we infer
\[
\tilde{\epsilon}_{2}(\zeta)\leq \vartheta(\zeta):=\frac{1}{2}-\frac{\zeta^{-4r}}{2(r+1)}.
\]
Numerical computations show $\vartheta(\zeta)$ improves the bound in Theorem~\ref{epsilont} for
$\zeta\in{I:=(1,2+\eta)}$ with a certain $\eta\in{(6\cdot 10^{-5},7\cdot 10^{-5})}$.
On the other hand, it is easy to check $\vartheta(\zeta)\in{(1/2-1/1024,1/2)}$ on the
entire interval $\zeta\in(1,\infty)$, and for $\zeta\in{I}$ even $\vartheta(\zeta)\in{(1/2-1/10368,1/2)}$,
so the improvement is small.
Also the lower bound $1/2-1/10368$ for $\vartheta(\zeta)$ can be attained up to arbitrarily small $\mu>0$
by taking $\zeta$ slightly larger than $\sqrt[3]{6}\approx 1.8171$.
Moreover $\vartheta(\zeta)$ obviously tends to $1/2$ as $\zeta$ tends to either $1$ or infinity.
\end{remark}

In fact we will prove a slight extension of Theorem~\ref{epsilont} in Section~\ref{cardgap}.
For algebraic numbers $\zeta>1$ we will further establish the following result concerning
$\tilde{\epsilon}_{1}(\zeta),\tilde{\epsilon}_{2}(\zeta)$. The proof of the first claim is
based on the properties of the Pisot numbers $\zeta_{m,b}$ carried out above, the second claim
follows from similar constructions we will present in Section~\ref{cardgap}.

\begin{theorem} \label{notintuitiver}
Let $m\geq 1$ be an integer and $\delta\in{(0,1)}$.
There exists a Pisot number {\upshape(}one may choose a unit{\upshape)} $\zeta$ of degree $m$ such that
$\frac{\delta}{\zeta-1}\leq \tilde{\epsilon}_{2}(\zeta)\leq \frac{1}{\zeta-1}$.
Moreover, there exists algebraic $\zeta>1$ of degree $m$ such that
$\frac{\delta}{2(\zeta-1)}\leq \tilde{\epsilon}_{1}(\zeta)\leq \frac{1}{2(\zeta-1)}$.
\end{theorem}

The first claim of Theorem~\ref{notintuitiver} is of particular interest because we will
carry out that we strongly expect (by a heuristic argument)
that for Lebesgue almost all $\zeta>1$ in fact $\frac{1}{2(\zeta-1)}$ is an upper
bound for $\tilde{\epsilon}_{2}(\zeta)$ as well. We will discuss this in Section~\ref{cardgap}.

For rational $\zeta$ we can further improve the bounds
from Theorem~\ref{epsilont}. As in~\cite{5}, for $z\in{\mathbb{R}}$ and $p/q$ rational in lowest terms let
\[
E(z):=\frac{1-(1-z)\prod_{m\geq 0} (1-z^{2^{m}})}{2z} , \qquad \tau(p/q):= \frac{E(q/p)}{p}.
\]
With this notation we have the following.

\begin{theorem} \label{gutkort}
Let $\zeta=p/q$ with integers $p>q\geq 2$ and $(p,q)=1$.
Then $\tilde{\epsilon}_{i}=\tilde{\epsilon}_{i}(p/q)$ for $i\in{\{1,2\}}$ satisfy
\[
\tau(p/q)\leq \tilde{\epsilon}_{1} \leq \min\left\{\frac{1}{2},\frac{q}{2(p-q)}\right\},  \quad
\max\left\{\tau(p/q),\frac{q}{2(p+q)}\right\}\leq \tilde{\epsilon}_{2}
\leq \min\left\{\frac{1}{2},\frac{q-1}{p-q}\right\}.
\]
In case of odd $q$, refined bounds are given by
\begin{equation} \label{eq:diebt}
\tilde{\epsilon}_{1}\leq \min\left\{\frac{1}{2},\frac{q-1}{2(p-q)}\right\},
\qquad   \max\left\{\tau(p/q),\frac{q+1}{2(p+q)}\right\} \leq \tilde{\epsilon}_{2}.
\end{equation}
In case of $q=2$, a refined bound for $\tilde{\epsilon}_{1}$ is
\[
\tilde{\epsilon}_{1}\leq \frac{1}{p}.
\]
\end{theorem}

The lower bound $\tau(p/q)$ at several places is due to Dubickas, the remaining bounds will be
settled in Section~\ref{cardgap}.
We point out another result for rational $\zeta$, which again we will compare to other
results and interpret in Section~\ref{cardgap}.

\begin{theorem} \label{hundertt}
Let $\zeta=p/q>1$ be a rational number but no integer.
If for $\alpha\neq 0$ and some large integer $n$
all numbers $\alpha(p/q)^{n},\alpha(p/q)^{n+1},\ldots$,$\alpha(p/q)^{n+l}$
lie in the interval $[-1/(p+q),1/(p+q)]$ mod $1$, then we have the asymptotic estimate
\[
l \leq n\cdot\left(\frac{\log p}{\log q}-1\right)+\log \vert\alpha\vert+o(1), \qquad n\to\infty.
\]

In particular $\varpi_{\epsilon,\zeta}=\emptyset$
for all $\epsilon\leq 1/(p+q)$ and thus $\tilde{\epsilon}_{1}(p/q)\geq 1/(p+q)$.
\end{theorem}

Finally we will derive the following results in the case that $\zeta$ is an integer.

\begin{theorem} \label{jep}
For an integer $\zeta=p/1>1$ we have
\[
\tilde{\epsilon}_{1}(\zeta)=\tilde{\epsilon}_{1}(p/1)=0, \qquad
\tau(p/1)\leq \tilde{\epsilon}_{2}(\zeta)\leq \frac{1}{p}-\frac{1}{p^{3}+p^{2}}.
\]
\end{theorem}

For example, for $p=10$ Theorem~\ref{jep} yields $0.099090099\cdots\leq \tilde{\epsilon}_{2}(10)\leq 0.099\overline{09}$.

\section{Preparatory cardinality results}   \label{kkk}

We will consider the situation of one fixed variable throughout the following.

\subsection{The case of fixed $\alpha$}  \label{fixedalpha}
We start with an easy proposition to simplify the proof
of Theorem~\ref{machet} later.

\begin{proposition} \label{bernardo}
Let $n$ be a positive integer, $x>3/2$ and $0<\epsilon<1/2$ be real numbers.
Then $(x+\epsilon)^{\frac{n+1}{n}}- (x-\epsilon)^{\frac{n+1}{n}}\geq 2\epsilon x^{\frac{1}{n}}$.
\end{proposition}

\begin{proof}
Define $\varphi_{n}: x\mapsto x^{\frac{n+1}{n}}$. We have to prove that
$\varphi_{n}(x+\epsilon)-\varphi_{n}(x-\epsilon)\geq 2\epsilon x^{\frac{1}{n}}$.
By Taylor expansion
$\varphi_{n}(x-\epsilon)= \varphi_{n}(x)-\epsilon \varphi_{n}^{\prime}(x)+
\frac{\epsilon^{2}}{2}\varphi_{n}^{\prime\prime}(\theta_{1})$
with some $\theta_{1}\in{(x-\epsilon,x)}$.
Similarly,
$\varphi_{n}(x+\epsilon)= \varphi_{n}(x)+\epsilon \varphi_{n}^{\prime}(x)+
\frac{\epsilon^{2}}{2}\varphi_{n}^{\prime\prime}(\theta_{2})$
with some $\theta_{2}\in{(x,x+\epsilon)}$.
Thus
\[
\varphi_{n}(x+\epsilon)-\varphi_{n}(x-\epsilon)= 2\frac{n+1}{n}\epsilon x^{\frac{1}{n}}+
\frac{\epsilon^{2}}{2}(\varphi_{n}^{\prime\prime}(\theta_{2})-\varphi_{n}^{\prime\prime}(\theta_{1}))
=2\epsilon x^{\frac{1}{n}}+\frac{1}{n}2\epsilon x^{\frac{1}{n}}
+\frac{\epsilon^{2}}{2}(\varphi_{n}^{\prime\prime}(\theta_{2})-\varphi_{n}^{\prime\prime}(\theta_{1})).
\]
We would certainly be done if the equivalent assertions
\begin{equation} \label{eq:unicorn}
\frac{1}{n}2\epsilon x^{\frac{1}{n}}>
\frac{\epsilon^{2}}{2}(\varphi_{n}^{\prime\prime}(\theta_{1})-\varphi_{n}^{\prime\prime}(\theta_{2})) \quad \Longleftrightarrow \quad
\frac{1}{n} x^{\frac{1}{n}}>
\frac{\epsilon}{4}(\varphi_{n}^{\prime\prime}(\theta_{1})-\varphi_{n}^{\prime\prime}(\theta_{2}))
\end{equation}
hold. We look at the right side of the equivalence.
The left hand side is obviously bounded below by $1/n$.
Applying Taylor Theorem again to the right hand side gives that right hand side
is bounded above by
$\vert\frac{\epsilon}{4}2\epsilon \varphi_{n}^{\prime\prime\prime}(\theta_{3})\vert
=\vert\frac{\epsilon^{2}}{2}\varphi_{n}^{\prime\prime\prime}(\theta_{3})\vert
\leq \frac{1}{8}\vert\varphi_{n}^{\prime\prime\prime}(\theta_{3})\vert$
with some $\theta_{3}\in{(\theta_{1},\theta_{2})\subseteq{(x-\epsilon,x+\epsilon)}}$.
However, $\vert\varphi_{n}^{\prime\prime\prime}(\theta_{3})\vert=
(\frac{n^{2}-1}{n^{3}})\theta_{3}^{-2+\frac{1}{n}}< \frac{1}{n}\theta_{3}^{-2+\frac{1}{n}}<\frac{1}{n}$
since $\theta_{3}>x-\epsilon>\frac{3}{2}-\frac{1}{2}=1$, proving \eqref{eq:unicorn}.
\end{proof}

\begin{definition}
For arbitrary real numbers $\alpha,\epsilon>0$, let $\chi_{\epsilon,\alpha}$ be
the set of all real $\zeta>1$ such that
$\Vert \alpha\zeta^{n}\Vert\leq \epsilon$ for all $n\geq n_{0}(\alpha,\zeta,\epsilon)$.
\end{definition}

Obviously $\chi_{\epsilon_{0},\alpha}\subseteq \chi_{\epsilon_{1},\alpha}$
for $\epsilon_{0}<\epsilon_{1}$ and any $\alpha$.
Note also that for $\lim_{n\to\infty}\Vert\alpha\zeta^{n}\Vert=0$,
the condition $\chi_{\epsilon,\alpha}\neq \emptyset$ for all $\epsilon>0$ is necessary.
In fact, for $\alpha$ fixed, the set $\bigcap_{\epsilon>0} \chi_{\epsilon,\alpha}$ coincides with
the set of values $\zeta$ such that $\lim_{n\to\infty}\Vert \alpha\zeta^{n}\Vert= 0$.
It is not hard to check that the sets $\chi_{\epsilon,\alpha}$ are closed under the maps
$\iota_{k}: \zeta\mapsto \zeta^{k}$ for $k\geq 1$ an integer.

The next Theorem~\ref{machet} is connected to Theorem~\ref{mosh}.
Given $\epsilon>0$, we explicitly construct intervals in which the
investigated set $\chi_{\epsilon,\alpha}$ of values $\zeta$ is dense or uncountable.
We point out in advance that it will turn out in Theorem~\ref{alzheimard} that indeed we do not
obtain uncountably many suitable values $\zeta$ in intervals of the form $(1,C)$ for sufficiently small $C$.
We restrict to the case of symmetric intervals with respect to $0$,
however the proof of this and most other results of Section~\ref{kkk}
easily extends to the more general case of arbitrary intervals of length $2\epsilon$,
see Remark~\ref{guteremark} and Remark~\ref{auchremark}. We remark that throughout the paper some results
stating that particular sets are uncountable use a method related to the one used by Pollington~\cite{pollington}
in Theorem~\ref{pollington}. A perspective for further research could be to provide more concise information
on Hausdorff dimensions of the involved sets.

\begin{theorem} \label{machet}
Let $\alpha,\epsilon>0$ be real numbers. The set
$\chi_{\epsilon,\alpha}\cap (1+\frac{1}{2\epsilon},\infty)$ is dense in $(1+\frac{1}{2\epsilon},\infty)$.
For any $a,b$ with $b>\max\{a,1+\frac{1}{\epsilon}\}$
the set $\chi_{\epsilon,\alpha}\cap (a,b)$ has cardinality of $\mathbb{R}$.
\end{theorem}

\begin{proof}
Fix $0<\epsilon<1/2$, which clearly is no restriction as the claim is trivial otherwise.
Moreover, we may assume $\alpha>0$.

Let $N_{0},n$ be any positive integers to be specified later such that
\begin{equation} \label{eq:megallan}
 (N_{0}-\epsilon)^{\frac{1}{n}}> \alpha^{\frac{1}{n}}\left(1+\frac{1}{2\epsilon}\right).
\end{equation}
Consider the interval
$I_{0}:=\left(\alpha^{-\frac{1}{n}}(N_{0}-\epsilon)^{\frac{1}{n}},\alpha^{-\frac{1}{n}}(N_{0}+\epsilon)^{\frac{1}{n}}\right)$.
By construction any $\zeta_{0}\in{I_{0}}$ satisfies $\alpha\zeta_{0}^{n}\in{J_{0}:=(N_{0}-\epsilon,N_{0}+\epsilon)}$.
Now by \eqref{eq:megallan} and Proposition \ref{bernardo} with $x:=N_{0}$, the interval
$K_{0}:=\left(\alpha^{-\frac{1}{n}}(N_{0}-\epsilon)^{\frac{n+1}{n}},
\alpha^{-\frac{1}{n}}(N_{0}+\epsilon)^{\frac{n+1}{n}}\right)$
has length at least $1+2\epsilon$.

Thus there exists an integer $N_{1}$ such that it contains
$J_{1}:=(N_{1}-\epsilon,N_{1}+\epsilon)$, so $J_{1}\subseteq K_{0}$. Putting
$I_{1}:=\left(\alpha^{-\frac{1}{n+1}}(N_{1}-\epsilon)^{\frac{1}{n+1}},
\alpha^{-\frac{1}{n+1}}(N_{1}+\epsilon)^{\frac{1}{n+1}}\right)$
we see that $I_{1}\subseteq I_{0}$ because by construction
$N_{1}\geq \alpha^{-\frac{1}{n}}(N_{0}-\epsilon)^{\frac{n+1}{n}}+\epsilon$
and hence
\begin{equation} \label{eq:vascogama}
 \alpha^{-\frac{1}{n+1}}(N_{1}-\epsilon)^{\frac{1}{n+1}}
\geq \alpha^{-\frac{1}{n+1}}\alpha^{-\frac{1}{n(n+1)}}(N_{0}-\epsilon)^{\frac{1}{n}}
= \alpha^{-\frac{1}{n}}(N_{0}-\epsilon)^{\frac{1}{n}},
\end{equation}
and similarly with inequality in reverse directions for the upper bounds of $I_{0},I_{1}$.
Combining \eqref{eq:megallan} and \eqref{eq:vascogama} yields in particular
\begin{equation} \label{eq:katzetot}
(N_{1}-\epsilon)^{\frac{1}{n+1}}>\alpha^{\frac{1}{n+1}}\left(1+\frac{1}{2\epsilon}\right).
\end{equation}
Furthermore, for any $\zeta_{1}\in{I_{1}}$ by construction
$\alpha\zeta_{1}^{n+1}\in{(N_{1}-\epsilon,N_{1}+\epsilon)}$.
So again by Proposition~\ref{bernardo} with $x:=N_{1}$ and \eqref{eq:katzetot}, if we similarly define
\[
K_{1}:=\left(\alpha^{-\frac{1}{n+1}}(N_{1}-\epsilon)^{\frac{n+2}{n+1}},
\alpha^{-\frac{1}{n+1}}(N_{1}+\epsilon)^{\frac{n+2}{n+1}}\right),
\]
the interval $K_{1}$ again has length at least $1+2\epsilon$.
Having now defined $I_{1},J_{1},K_{1}$ we can repeat the procedure to obtain $J_{2},I_{2},K_{2}$
in this succession with very similar properties.

Proceeding in this manner gives a sequence of nested intervals $I_{1}\supseteq I_{2}\supseteq I_{3}\cdots$.
Defining $\zeta:=\cap_{j\geq 0} I_{j}$, which clearly is a unique real number,
it is easy to see $\zeta$ has the desired property $\vert \alpha\zeta^{n+j}-N_{j}\vert \leq \epsilon$
for all $j\geq 0$.

To see $\chi_{\epsilon,\alpha}$ is dense in $(1+\frac{1}{2\epsilon},\infty)$, we need to show for
fixed $0<\epsilon<1/2$ and any given $d>c>1+\frac{1}{2\epsilon}$,
for some pair $(N_{0},n)$ satisfying \eqref{eq:megallan}
the $\zeta$ arising by the above construction has property $\zeta\in{(c,d)}$.
Indeed, it suffices to take any integer
$N_{0}\in{\left(\frac{c^{n}}{\alpha}+1,\frac{d^{n}}{\alpha}-1\right)}$ for $n$ sufficiently large
that the interval is non-empty, to guarantee $\zeta=\cap_{j\geq 0} I_{j}\subseteq I_{0}\subseteq (c,d)$ for
the resulting $\zeta$ as well as the condition \eqref{eq:megallan}.

To see $\chi_{\epsilon,\alpha}$ has cardinality of $\mathbb{R}$ in any non-empty
interval $(a,b)$ with $b>1+\frac{1}{\epsilon}$,
repeat the above construction with $(N_{0}-\epsilon)^{\frac{1}{n}}> \alpha^{\frac{1}{n}}(1+\frac{1}{\epsilon})$
instead of $\alpha^{\frac{1}{n}}(1+\frac{1}{2\epsilon})$, and observe that the resulting intervals $K_{j}$
have length at least $2+2\epsilon$. So we have the choice of at least two different values $N_{j}$
in each step. Different choices of $N_{j}$ by construction induce disjoint intervals $I_{j+1}$ in the next step,
so the intersections $\cap_{j\geq 0} I_{j}$ do not coincide for any two different choices as well.
Hence the set has cardinality of the power set of $\mathbb{N}$ which equals cardinality of $\mathbb{R}$,
and by a similar argument as above we may choose $I_{0}$ to
lie in any given interval $(a,b)$ with $b>a>1+\frac{1}{\epsilon}$.
Thus $\chi_{\epsilon,\alpha}\cap (a,b)$ has cardinality of $\mathbb{R}$.
\end{proof}

\begin{remark}
Note that the interval bounds in Theorem~\ref{machet} do not depend on $\alpha$. Moreover,
reviewing the proof, in fact the minimal $n=n_{0}(\alpha,\zeta,\epsilon)$ in the construction
for $\zeta$ in a given interval $\zeta\in{(c,d)}$ only depends on $c,d$, and the condition becomes
weaker with growing $c$ and $d-c$. Thus we may write $n\geq n_{0}(\alpha,\epsilon,d-c)$ for
all $\zeta\in{\chi_{\epsilon,\alpha}\cap(c,d)}$.
\end{remark}

\begin{remark} \label{guteremark}
The proof can be readily extended to the case where $\{\alpha\zeta^{n}\}$ lie in arbitrary
closed intervals $I_{n} \bmod 1$ of length $2\epsilon$. The same will apply to Theorem~\ref{alzheimard}.
\end{remark}

The proof of Theorem~\ref{machet} suggests that for all $\alpha\neq 0$, or at least
almost all $\alpha$ in the sense of Lebesgue measure,
in fact $\chi_{\epsilon,\alpha}\cap (1+\frac{1}{2\epsilon},\infty)$ should be uncountable.
Assume otherwise for some $\alpha\neq 0$ the set $\chi_{\epsilon,\alpha}\cap (1+\frac{1}{2\epsilon},\infty)$
is at most countable. Then starting with a pair $N_{0},n$ satisfying \eqref{eq:megallan},
the recursive process would yield only one integer in the intervals $K_{j}$ for
all large $j$ (else we have $2$ choices infinitely often, contradicting the assumption).
The intervals $K_{j}$ have length greater than $1+2\epsilon$, so this means
their midpoints avoid a neighborhood of $1/2$. It is reasonable to believe that this biased distribution leads
to a set of values $\alpha$ of measure $0$ for the fixed pair $N_{0},n$, see also Theorem~\ref{kok}.
Note that this must hold for any pair $N_{0},n$ satisfying \eqref{eq:megallan}.
A rigorous proof seems hard, however. We will carry out a similar phenomenon
in a preciser way in Section~\ref{fixedzeta}, see in particular Proposition~\ref{lemur}.

As indicated previous to Theorem~\ref{machet}, the set $\chi_{\epsilon,\alpha}$ is reasonably
smaller for $\zeta$ in a neighborhood of $1$.

\begin{theorem} \label{alzheimard}
For any fixed $\alpha\neq 0,\epsilon>0$, the set
$\chi_{\epsilon,\alpha}\cap (1,\frac{1}{2\epsilon}-1)$ is at most countable.
\end{theorem}

\begin{proof}
By definition, if $\zeta$ lies in $\chi_{\epsilon,\alpha}$
there exists an integer sequence $(N_{n})_{n\geq 1}$ such that
$\alpha\zeta^{n}\in{[N_{n}-\epsilon,N_{n}+\epsilon]}$ for $n\geq n_{0}=n_{0}(\zeta,\epsilon,\alpha)$.
By $\alpha\zeta^{n+1}=(\alpha\zeta^{n})^{\frac{n+1}{n}}\alpha^{-\frac{1}{n}}$ we infer
\begin{equation} \label{eq:narrensicher}
\alpha^{-\frac{1}{n}}(N_{n}-\epsilon)^{\frac{n+1}{n}}\leq \alpha\zeta^{n+1}
\leq \alpha^{-\frac{1}{n}}(N_{n}+\epsilon)^{\frac{n+1}{n}}.
\end{equation}
Suppose we have already shown
\begin{equation} \label{eq:soidberge}
\alpha^{-\frac{1}{n}}\left((N_{n}+\epsilon)^{\frac{n+1}{n}}-(N_{n}-\epsilon)^{\frac{n+1}{n}}\right)
<1-2\epsilon, \qquad n\geq n_{0}.
\end{equation}
Then clearly there is at most one integer $N_{n+1}$ such that
\[
 [N_{n+1}-\epsilon,N_{n+1}+\epsilon]\bigcap
\left(\alpha^{-\frac{1}{n}}(N_{n}-\epsilon)^{\frac{n+1}{n}},
\alpha^{-\frac{1}{n}}(N_{n}+\epsilon)^{\frac{n+1}{n}}\right) \neq \emptyset.
\]
By \eqref{eq:narrensicher} the property
$\alpha\zeta^{n+1}\in{[N_{n+1}-\epsilon,N_{n+1}+\epsilon]}$ is valid
for $n\geq n_{0}$. As this is true for $n+2,n+3,\ldots$ with the same argument,
the sequence $(N_{n})_{n\geq n_{0}}$ and hence $\zeta$ is determined
by some $n_{0}=n_{0}(\epsilon,\alpha,\zeta),N_{n_{0}}$.
However, the sequence $(N_{n})_{n\geq n_{0}}$ clearly determines a unique $\zeta$
because obviously $\zeta=\lim_{n\to\infty} \sqrt[n]{N_{n}/\alpha}=\lim_{n\to\infty} \sqrt[n]{N_{n}}$.
Thus $\zeta\mapsto (n_{0},N_{n_{0}})$ induces a one-to-one map
from $\chi_{\epsilon,\alpha}$ to $\mathbb{Z}^{2}$ and hence the set is
at most countable. Hence it only remains to prove \eqref{eq:soidberge}.

Recall the functions $\varphi_{n}$ from the proof of Proposition~\ref{bernardo}.
We have $\varphi_{n}(x+\epsilon)-\varphi_{n}(x-\epsilon)= 2\epsilon\varphi_{n}^{\prime}(\theta)$
for some $\theta\in{(x-\epsilon,x+\epsilon)}$. Hence the left hand side of \eqref{eq:soidberge}
is bounded above by $\alpha^{-\frac{1}{n}}\frac{n+1}{n}2\epsilon(N_{n}+\epsilon)^{\frac{1}{n}}$.
Clearly $\lim_{n\to\infty} \alpha^{-\frac{1}{n}}\frac{n+1}{n}=1$, and
as $\sqrt[n]{N_{n}}$ tends to $\zeta$ so does $\sqrt[n]{N_{n}+\epsilon}$ for fixed $\epsilon$.
Claim \eqref{eq:soidberge} follows from our assumption $\zeta<\frac{1}{2\epsilon}-1$.
\end{proof}

We compare our result with~\cite[Theorem~3.5]{29} concerning the distribution
of $\Vert\alpha\zeta^{n}-\theta_{n}\Vert$ for an arbitrary given sequence $(\theta_{n})_{n\geq 1}$.

\begin{theorem}[Lerma, part 1] \label{lerma}
For any $\alpha\neq 0$ and $A>1$ and any given sequence $(r_{n})_{n\geq 1}$,
there exists $\zeta$ such that
\[
A \leq \zeta \leq A+\frac{A}{(A-1)\vert \alpha\vert}
\]
and for every $n\geq 1$
\[
 \Vert \alpha\zeta^{n}-r_{n}\Vert \leq \frac{1}{2(A-1)}.
\]
\end{theorem}

Putting $r_{n}=0$ for all $n\geq 1$ and restricting to $\alpha>0$
and identifying $\epsilon$ with $\frac{1}{2(A-1)}$ in Theorem~\ref{lerma} implies
the existence of $\zeta$ with
$\frac{1}{2\epsilon}+1\leq \zeta \leq \frac{1}{2\epsilon}+1+\frac{1+2\epsilon}{\alpha}$
such that $\Vert \alpha\zeta^{n}\Vert \leq \epsilon$. Thus, the generalized result of Theorem~\ref{machet}
pointed out in Remark~\ref{guteremark}, contains more information than Theorem~\ref{lerma}.

\subsection{The case of fixed $\zeta$} \label{fixedzeta}

Now we want to study the reverse situation, i.e. for $\zeta>1$ and $0<\epsilon< 1/2$ fixed we
ask which $\alpha$ satisfy $\left\Vert \alpha\zeta^{n}\right\Vert\leq \epsilon$ for all large $n$.
This is the setup for all the results from Section~\ref{outline}.
Recall Definition~\ref{wdeff} for the present section.

\begin{theorem} \label{muskatottonel}
For any $\epsilon>0$ and $\zeta\geq 1+\frac{1}{2\epsilon}$, the set $\varpi_{\epsilon,\zeta}$ is dense in $\mathbb{R}$.
If $\zeta\geq 1+\frac{1}{\epsilon}$, the set $\varpi_{\epsilon,\zeta}\cap (a,b)$ has cardinality of $\mathbb{R}$ for
any $b>a$. Numbers in $\varpi_{\epsilon,\zeta}$ can be recursively constructed.
\end{theorem}

\begin{proof}
We may assume $\alpha>0$.
For any fixed $\zeta,\epsilon,c,d$ with $\zeta\geq 1+\frac{1}{2\epsilon}$ and $d>c>0$,
we must prove there is $\alpha\in{(c,d)\cap \varpi_{\epsilon,\zeta}}$.
Take $n_{0}=n_{0}(\epsilon,\zeta)$ sufficiently large that $(d-c)\zeta^{n_{0}}>1+2\epsilon$.
Then there exists an integer $N_{0}$ such that $[N_{0}-\epsilon,N_{0}+\epsilon]\subseteq (c\zeta^{n_{0}},d\zeta^{n_{0}})$.
Let $I_{0}:=[\frac{N_{0}-\epsilon}{\zeta^{n_{0}}},\frac{N_{0}+\epsilon}{\zeta^{n_{0}}}]$, then any $\alpha\in{I_{0}}$
satisfies $\alpha\zeta^{n_{0}}\in{[N_{0}-\epsilon,N_{0}+\epsilon]}$.
By assumption $2\epsilon\cdot\zeta\geq 1+2\epsilon$, so there exists some integer $N_{1}$ with
$[N_{1}-\epsilon,N_{1}+\epsilon]\subseteq \zeta^{n_{0}+1}I_{0}$. Defining
$I_{1}:=[\frac{N_{1}-\epsilon}{\zeta^{n_{0}}},\frac{N_{1}+\epsilon}{\zeta^{n_{0}}}]$, any $\alpha$ in $I_{1}$
satisfies $\alpha\zeta^{n+1}\subseteq [N_{1}-\epsilon,N_{1}+\epsilon]$. Moreover $I_{1}\subseteq I_{0}$.
Proceeding in this manner gives a nested sequence $(c,d)\supseteq I_{0}\supseteq I_{1}\supseteq I_{2}\supseteq \cdots$
of intervals, which intersect in a single point $\alpha_{0}:=\bigcap_{j\geq 0} I_{j}$
because the length of $I_{j}$ is $\frac{2\epsilon}{\zeta^{n_{0}+j}}$ which tends to zero. For this
$\alpha_{0}$ indeed we have both $\alpha_{0}\in{(c,d)}$ and $\Vert\alpha_{0}\zeta^{n}\Vert \leq \epsilon$
for any $n\geq n_{0}$.

The cardinality argument is very similar to that in the proof of Theorem~\ref{machet}, using that
by the assumption $2\epsilon\cdot\zeta\geq 2+2\epsilon$ we have at least two choices for $N_{j}$ in any step.
\end{proof}

We point out that
the proof of Theorem~\ref{muskatottonel} suggests that for almost all fixed $\zeta>1$,
the property $\zeta> 1+\frac{1}{2\epsilon}$ or equivalently $\epsilon>\frac{1}{2(\zeta-1)}$
should suffice for $\varpi_{\epsilon,\zeta}$ to be uncountable.
Roughly speaking, assuming a not too biased distribution of $\{\zeta N_{j}\}$ in $(0,1)$ for $N_{j}$ as in the proof
of Theorem~\ref{muskatottonel}, will be sufficient for $\varpi_{\epsilon,\zeta}$ to be uncountable.
Proposition~\ref{lemur} will state this observation in a preciser way. We introduce
some identifications used in its proof and in fact carry out the essential parts of the proof beforehand.

Start with any integer $N_{0}$.
Proceed as in the proof of Theorem~\ref{muskatottonel} with the recursive construction of $N_{j}$.
Concretely, consider the interval $I_{1}=\zeta\cdot[N_{0}-\epsilon,N_{0}+\epsilon]$ and consider the integers
$N_{1}$ for which $[N_{1}-\epsilon,N_{1}+\epsilon]\subseteq I_{1}$. For any such $N_{1}$ repeat this process
and so on. As used in the proof, the assumption $\zeta> 1+\frac{1}{2\epsilon}$
is equivalent to $2\epsilon\cdot\zeta> 1+2\epsilon$.
Thus there is at least one $N_{j+1}$ in any step, and
the strict inequality means that one would expect that with fixed positive probability
there should be (at least) two.
This is the case if the midpoint of the interval $\zeta\cdot[N_{j}-\epsilon,N_{j}+\epsilon]$,
that is $N_{j}\zeta$, has fractional part in the fixed neighborhood
$[1-\epsilon\zeta+\epsilon,\epsilon\zeta-\epsilon]\neq \emptyset$ of $1/2$.
The process can be viewed as an infinite tree $\mathscr{T}=\mathscr{T}(\zeta,\epsilon,N_{0})$ with (multiply defined)
vertices $N_{j}$ and root $N_{0}$ in which any vertex apart from $N_{0}$ has precisely one ancestor vertex
and any vertex has at least one successor vertex.
Any infinite path $N_{0},N_{1},\ldots$ corresponds to an element of $\varpi_{\epsilon,\zeta}$
and this assignment is injective, as established in the proof of Theorem~\ref{muskatottonel}.
We will identify any path in $\mathscr{T}$  with the induced element in $\varpi_{\epsilon,\zeta}$.
Call a path in $\mathscr{T}$ {\em deterministic} if it
contains some vertex $N_{j}$ for which there is no other path in $\mathscr{T}$
starting with the same initial vertex sequence $N_{0},N_{1},\ldots,N_{j}$. If $N_{j}$ is
such a vertex say the path is {\em deterministic for $N_{j}$}. Clearly if a path
is deterministic for $N_{j}$ then it is also deterministic for all successor vertices $N_{j+1},N_{j+2},\ldots$.
Observe that if $\mathscr{T}$ contains no deterministic path,
the set of paths and thus $\varpi_{\epsilon,\zeta}$ is uncountable.
Indeed, if there is no deterministic path,
deleting the vertices of the tree where we have only one choice and reducing the number of paths
in the remaining tree if necessary by cutting off,
leads to the classical infinite binary tree, say $\mathscr{T}_{2}$.
This procedure clearly induces a surjective map from the paths of $\mathscr{T}$
to those of $\mathscr{T}_{2}$. Since there are uncountably many paths in $\mathscr{T}_{2}$,
as the binary expansion of any element of $(0,1)$ can be obtained by going to
the left is reading the digit $0$ and to the right $1$, the claim follows.
Obviously, the above argument can be extended to show that if $\varpi_{\epsilon,\zeta}$ is only
countable, then for any path in $\mathscr{T}$ and arbitrary large $j_{0}$, there exists
a path in $\mathscr{T}$ deterministic for some $N_{j}$ with $j\geq j_{0}$ with coinciding initial vertex
sequence $N_{0},N_{1},\ldots,N_{j_{0}}$.
Moreover, if a path is deterministic for $N_{j_{0}}$ then $N_{j+1}=\scp{\zeta N_{j}}$ for $j\geq j_{0}$
by construction.
However, note that $\varpi_{\epsilon,\zeta}$ being at most countable does not necessarily mean
any path in any corresponding tree $\mathscr{T}(\zeta,\epsilon,N_{0})$ with an integer parameter $N_{0}$
must be deterministic.
Define a binary tree $\mathscr{T}^{\prime}$ with root $N_{0}^{\prime}$ by the rule that going to the right
induces a deterministic path by having to go to the right in every further step,
but going to the left allows both directions
in the following step. The set of paths of $\mathscr{T}^{\prime}$, corresponding to elements of
$\varpi_{\epsilon,\zeta}$, is countable but the path defined by going to the left in every step is not
deterministic for any $N_{j}^{\prime}$.

\begin{definition} \label{exceptional}
Call $\zeta>1$ {\em exceptional} if and only if
for some $\epsilon>\frac{1}{2(\zeta-1)}$ the set $\varpi_{\epsilon,\zeta}$ is at most countable.
Let $\Theta\subseteq(1,\infty)$ be the set of exceptional numbers.
\end{definition}

In fact $\Theta\subseteq (2,\infty)$ since $\zeta\leq 2$ implies the trivial bound $\epsilon> 1/2$.
Let $\mathbb{N}=\{1,2,\ldots\}$.

\begin{definition} \label{V}
For real $\zeta$ and every $N_{0}\in\mathbb{N}$, define the sequence $(N_{j})_{j\geq 0}$
recursively by $N_{j+1}=\scp{\zeta N_{j}}$ for $j\geq 0$.
Let $W(\zeta)\subseteq \mathbb{N}$ be the set of integers $N_{0}$ for which the
corresponding sequence $(\{N_{j}\zeta\})_{j\geq 0}$ of fractional parts is not dense
at $1/2$. Let $\Gamma\subseteq{(3/2,\infty)}$ consist of the numbers $\zeta>3/2$ for which
$W(\zeta)\neq \emptyset$.
\end{definition}

The lower bound $3/2$ instead of $1$ is only necessary to ensure $N_{j+1}>N_{j}$ in order to
avoid constant sequences $(N_{j})_{j\geq 0}$ which would lead to unintended elements $\zeta\in{\Gamma}$.
Alternatively one could ask for $W(\zeta)$ to be infinite instead of non-empty.
It will turn out not to be of importance anyway since by
the above remark $\Theta\subseteq (2,\infty)$ we may restrict to the interval $(2,\infty)$ for our purposes.

\begin{proposition} \label{lemur}
We have $\Theta\subseteq \Gamma$. In particular, if
$\Gamma$ has Lebesgue measure $0$, then so has $\Theta$.
\end{proposition}

\begin{proof}
Assume $\zeta>1$ is exceptional, that is for some $\epsilon>0$ with
$\zeta>1+\frac{1}{2\epsilon}$, the set $\varpi_{\epsilon,\zeta}$ is only countable.
For any positive integer $N_{0}$ consider
the arising tree $\mathscr{T}$ as carried out above.
In view of the preceding remarks $\mathscr{T}$ contains a deterministic path,
i.e. a path $(N_{j})_{j\geq 0}$ of $\mathscr{T}$ with the property
that for some integer $j_{0}$ there is no other path in $\mathscr{T}$ whose initial
vertex sequence coincides with $N_{0},N_{1},\ldots,N_{j_{0}}$.
Fixing this path, we can treat $j_{0},N_{j_{0}}$ as fixed too.
As carried out above, for $j\geq j_{0}$ all fractional parts of $\{N_{j}\zeta\}$
of the induced sequence $(N_{j})_{j\geq j_{0}}$ must avoid the fixed symmetric neighborhood
$[1-\epsilon\zeta+\epsilon,\epsilon\zeta-\epsilon]\neq \emptyset$ of $1/2$.
Hence we have found a path with $(\{N_{j}\zeta\})_{j\geq 0}$ not dense at $1/2$.
Since $N_{j+1}=\scp{\zeta N_{j}}$ for $j\geq j_{0}$,
we deduce $N_{j_{0}}\in{W(\zeta)}$ and $\zeta$ belongs to $\Gamma$.
Since $\zeta\in{\Theta}$ was arbitrary the claim follows.
\end{proof}

If we write $\epsilon=\delta\cdot\frac{1}{\zeta-1}$ for the largest $\epsilon$ in Definition~\ref{exceptional},
then $\delta\in{(1/2,1]}$ by Theorem~\ref{muskatottonel}.
Larger $\delta$ implies a larger symmetric interval $I=[1-\epsilon\zeta+\epsilon,\epsilon\zeta-\epsilon]$ around
$1/2$ without any number $\{N_{j}\zeta\}$ in $I$ for large $j$ where $N_{j}=\scp{\alpha\zeta^{j}}$,
with $I=[0,1]$ if $\delta=1$.
By sigma additivity of the Lebesgue measure, for the proof of the hypothesis of Proposition~\ref{lemur},
it suffices to show that for any {\em fixed} $N_{0}\geq 1$ the set of $\zeta>1$ with $(\{N_{j}\zeta\})_{j\geq 1}$
not dense at $1/2$ has measure $0$. Hence, if we dropped
the rounding to the next integer in any step, that is $N_{j+1}=\zeta N_{j}$, then
it would follow from Theorem~\ref{kok} that almost all $\zeta>1$ are not in $\Gamma$ and thus not exceptional.
Having ruled out the case of constant sequences by the assumption $\zeta>3/2$,
there is no reason why the rounding should affect this result,
however a rigorous proof seems hard.
On the other hand, Theorem~\ref{mosh} and Theorem~\ref{pollington}
strongly suggest that $\Gamma$ has full dimension.

In fact, we have shown in Proposition~\ref{lemur} that for $\zeta\in{\Theta}$,
for any start value $N_{0}$ the recursive process starting at $N_{0}$
becomes determined for most choices of paths.
However, observe that for $\zeta\in{\mathbb{N}_{\geq 2}}$
the worst case in the construction, that is all $\zeta N_{j}$ are
integers, applies. Hence $\mathbb{N}_{\geq 2}\subseteq \Gamma$.
A result due to Dubickas implies that
$\mathbb{N}_{\geq 3}\subseteq \Theta$, see Section~\ref{sektions} and also Theorem~\ref{jep}.
Moreover,
any $\zeta>1$ for which there exists $\alpha\neq 0$ such that $\lim_{n\to\infty}\Vert \alpha\zeta^{n}\Vert=0$,
in particular any Pisot number, belongs to $\Gamma$. Indeed it is easily checked that
in this case $\scp{\alpha\zeta^{n}}\in{W(\zeta)}$ for any sufficiently large $n$.
In fact for sufficiently large $j$ the corresponding fractional parts $\{N_{j}\zeta\}$ lie in
arbitrarily short intervals with midpoint $0$ modulo $1$. We will see in Section~\ref{sektion3} that at least some
Pisot numbers of any given degree are exceptional, which is a little surprising considering that we can start
the above process at any $N_{0}\geq 1$.
Another interesting special case is $\zeta=p/q$ rational but not an integer.
We will treat it in Section~\ref{sektions}.

A result somehow reverse to Theorem~\ref{muskatottonel} is the following.

\begin{theorem} \label{machete}
Let $\zeta>1,\epsilon>0$ be real numbers with $(\zeta+1)\epsilon<1/2$.
Then the set $\varpi_{\epsilon,\zeta}$ is at most countable. Unless $\zeta$ is rational with even
denominator in lowest terms, it suffices to assume $(\zeta+1)\epsilon\leq 1/2$.
\end{theorem}

\begin{proof}
Let $n_{0}=n_{0}(\alpha,\zeta,\epsilon)$ be an integer with the above property
for fixed $\zeta,\epsilon,\alpha$ as in the theorem. For $\alpha$ to satisfy
the assertions it is obvious that
\[
 \alpha\in{\bigcap_{n\geq n_{0}} I_{n}},
\qquad I_{n}:=\left[\frac{M_{n}-\epsilon}{\zeta^{n}},\frac{M_{n}+\epsilon}{\zeta^{n}}\right]
\]
for some integer sequence $(M_{n})_{n\geq n_{0}}$.
Obviously, in this case
\begin{equation} \label{eq:gleihammas}
\alpha= \bigcap_{n\geq n_{0}} I_{n}=\lim_{n\to\infty} \frac{M_{n}}{\zeta^{n}}.
\end{equation}
For $\alpha\in{I_{n}\cap I_{n+1}}$ it is necessary that $I_{n},I_{n+1}$ are not disjoint which requires
\[
 \left\vert \frac{M_{n}}{\zeta^{n}}-\frac{M_{n+1}}{\zeta^{n+1}}\right\vert
 \leq \frac{\epsilon}{\zeta^{n}}+\frac{\epsilon}{\zeta^{n+1}}.
\]
This is equivalent to $\vert \zeta M_{n}-M_{n+1}\vert\leq (\zeta+1)\epsilon$.
By the assumption $(\zeta+1)\epsilon<1/2$ this means $M_{n+1}=\scp{\zeta M_{n}}$
is uniquely determined by $M_{n}$. The same holds if $\zeta$ is irrational
(or rational with odd denominator) and $(\zeta+1)\epsilon\leq 1/2$, since then clearly $\{M_{n}\zeta\}\neq 1/2$.
This holds for any $n\geq n_{0}$, so $M_{n_{0}}$ determines the sequence $(M_{n})_{n\geq n_{0}}$
and hence $\alpha$ by \eqref{eq:gleihammas}. However, for any fixed $\alpha\in{\varpi_{\epsilon,\zeta}}$
there is a $n_{0}=n_{0}(\alpha,\zeta,\epsilon)$ such that the above holds with some $M_{n_{0}}$.
So $\alpha\mapsto (n_{0},M_{n_{0}})$ induces a one-to-one map from $\varpi_{\epsilon,\zeta}$
to $\mathbb{Z}^{2}$, which means that $\varpi_{\epsilon,\zeta}$ is at most countable.
\end{proof}

\begin{remark} \label{auchremark}
The analogue of Remark~\ref{guteremark} holds for Theorem~\ref{muskatottonel} and,
apart from the equality case, for Theorem~\ref{machete} for the same reasons.
Moreover, Proposition~\ref{lemur} essentially holds for an arbitrary fixed interval modulo $1$ of
length $2\epsilon$ instead of the $0$-symmetric one in $\varpi_{\epsilon,\zeta}$,
where $1/2$ in the definition of $\Gamma$ must be replaced by some other value.
\end{remark}

Comparing Theorem~\ref{machete} to Theorem~\ref{bertin}, we see for fixed $\zeta$ our
bound is better in view of the square, however it is not uniform in $\zeta$ as
Theorem~\ref{bertin}. In comparison to our results we quote the second assertion of~\cite[Theorem~3.5]{29}.

\begin{theorem}[Lerma, part 2] \label{mlerma}
For any $\zeta>1, L\neq 0$ and any given sequence $(r_{n})_{n\geq 1}$,
there exists $\alpha$ such that
\[
\vert L\vert \leq \vert \alpha\vert \leq \vert L\vert + \frac{1}{\zeta-1}
\]
and for every $n\geq 1$
\[
\Vert \alpha\zeta^{n}-r_{n}\Vert \leq \frac{1}{2(\zeta-1)}.
\]
\end{theorem}

This implies $\vert\varpi_{\epsilon,\zeta}\vert \geq \vert \mathbb{Z}\vert$ for
$\epsilon\geq \frac{1}{2(\zeta-1)}$, which is nontrivial provided $\zeta>2$.
This bound coincides with our bound from Theorem~\ref{muskatottonel}, which can
again be generalized to arbitrary sequences $(r_{n})_{n\geq 1}$
as in Theorem~\ref{mlerma}, as indicated in Remark~\ref{auchremark}.
Thus (the generalized) Theorem~\ref{muskatottonel} implies Theorem~\ref{mlerma}.

\section{The cardinality gap phenomenon} \label{cardgap}

Now we turn to the main focus of the paper,
that is to investigate what we will call the cardinality gap phenomenon.
Roughly speaking it means to investigate for which
parameters the sets defined in Section~\ref{kkk} are countable versus uncountable.
The following Corollary~\ref{kor1} should portray the spirit of cardinality gap phenomena more accurately.

\subsection{Fixed $\epsilon$} \label{epsfix}

In the present section we agree on $\sup\{\emptyset\}=1$ in order to
formulate some results in widest generality (taking care of rather large $\epsilon$).
We point out the observed cardinality gap arising from Theorem~\ref{machet}
and Theorem~\ref{alzheimard} as a corollary.

\begin{corollary}  \label{kor1}
Let $\alpha\neq 0$ be fixed. For any $\epsilon>0$ define $\zeta_{1}=\zeta_{1}(\epsilon)$ by
\begin{equation} \label{eq:olympb}
\zeta_{1}:=\sup \left\{C>1: \left\vert\chi_{\epsilon,\alpha}\cap (C,\infty)\right\vert\leq \vert \mathbb{Z}\vert\right\}
=\inf \left\{C>1: \left\vert\chi_{\epsilon,\alpha}\cap (C,\infty)\right\vert>\vert \mathbb{Z}\vert\right\}.
\end{equation}
Similarly, define $\zeta_{2}=\zeta_{2}(\epsilon)$ by
\begin{equation} \label{eq:ierade}
\zeta_{2}:=\sup \left\{C>1: \left\vert\chi_{\epsilon,\alpha}\cap (a,b)\right\vert\leq \vert \mathbb{Z}\vert\right\}
=\inf \left\{C>1: \left\vert\chi_{\epsilon,\alpha}\cap (a,b)\right\vert>\vert \mathbb{Z}\vert\right\}
\end{equation}
where we understand the above to hold simultaneously for all intervals $(a,b)\subseteq (C,\infty)$.
Then $\zeta_{1}\in{\left[\max\{1,\frac{1}{2\epsilon}-1\},1+\frac{1}{\epsilon}\right]}$
and $\zeta_{2}\in{[\zeta_{1},1+\frac{1}{\epsilon}]}
\subseteq{\left[\max\{1,\frac{1}{2\epsilon}-1\},1+\frac{1}{\epsilon}\right]}$.
\end{corollary}

\begin{remark}
It would be nice to have cardinality equal to $\vert \mathbb{R}\vert$ instead of greater $\vert\mathbb{Z}\vert$
on the right hand sides in \eqref{eq:olympb}, \eqref{eq:ierade}.
If we assume the continuum hypothesis to be true (which is known to be undecidable due to P. Cohen),
then indeed we may make this replacement. However, if we do not assume that it is true,
the convenient values $\zeta_{1},\zeta_{2}$ might not be well-defined any more with the replacement.
Related issues will apply frequently in similar situations the sequel.
\end{remark}

Note that no set $\chi_{\epsilon,\alpha}\cap (C,\infty)$ and thus $\chi_{\epsilon,\alpha}$ cannot
be finite unless it is empty, since $\chi_{\epsilon,\alpha}$ is closed
under any map $\iota_{k}$ defined in Section~\ref{fixedalpha}.
However, it is not clear if $\chi_{\epsilon,\alpha}$ can have isolated points.
By Theorem~\ref{machet} and Theorem~\ref{alzheimard} isolated points can only occur in
the interval $(1,\frac{1}{2\epsilon}+1)$.
One may further ask whether there can be finitely many equivalence classes under
the equivalence relation $\zeta_{1}\thicksim \zeta_{2} \Leftrightarrow \zeta_{1}^{p}=\zeta_{2}^{q}$
for positive integers $p,q$.

Similarly, we infer a cardinality gap corollary from Theorem~\ref{muskatottonel} and Theorem~\ref{machete}.

\begin{corollary}  \label{kor2}
 For any $\epsilon>0$, define
$\tilde{\zeta}_{1}=\tilde{\zeta}_{1}(\epsilon)$ by
\[
\tilde{\zeta}_{1}=\sup \left\{\zeta>1: \left\vert\varpi_{\epsilon,\zeta}\right\vert\leq \vert \mathbb{Z}\vert\right\}
=\inf \left\{\zeta>1: \left\vert\varpi_{\epsilon,\zeta}\right\vert>\vert \mathbb{Z}\vert\right\}.
\]
Similarly, for fixed real numbers $b>a$ define $\tilde{\zeta}_{2}=\tilde{\zeta}_{2}(\epsilon,a,b)$ by
\[
\tilde{\zeta}_{2}=
\sup \left\{\zeta>1: \left\vert\varpi_{\epsilon,\zeta}\cap(a,b)\right\vert\leq \vert \mathbb{Z}\vert\right\}
=\inf \left\{\zeta>1: \left\vert\varpi_{\epsilon,\zeta}\cap(a,b)\right\vert>\vert \mathbb{Z}\vert\right\}.
\]
Then $\tilde{\zeta}_{1}\in{\left[\max\{1,\frac{1}{2\epsilon}-1\},1+\frac{1}{\epsilon}\right]}$
and
$\tilde{\zeta}_{2}\in{[\tilde{\zeta}_{1},1+\frac{1}{\epsilon}]}\subseteq
{\left[\max\{1,\frac{1}{2\epsilon}-1\},1+\frac{1}{\epsilon}\right]}$.
\end{corollary}

Note that again for given $\zeta>1,\epsilon>0$, the assumption $\varpi_{\epsilon,\zeta}\neq \emptyset$
is equivalent to $\vert\varpi_{\epsilon,\zeta}\vert\geq \vert \mathbb{Z}\vert$, since
$\varpi_{\epsilon,\zeta}$ is closed under the maps $\tau_{k,\zeta}$ defined in Section~\ref{fixedzeta}.
One may ask whether this is true for $\varpi_{\epsilon,\zeta}\cap (a,b)$ as well.
Moreover, one may ask whether the number of residue classes of $\varpi_{\epsilon,\zeta}$
under certain equivalence relations is finite. For example
\[
\alpha_{1}\thicksim_{1} \alpha_{2} \Leftrightarrow \frac{\alpha_{2}}{\alpha_{1}}=\zeta^{l},\qquad
\alpha_{1}\thicksim_{2} \alpha_{2} \Leftrightarrow \alpha_{2}=\alpha_{1}^{m/n}, \qquad
\alpha_{1}\thicksim_{3} \alpha_{2} \Leftrightarrow \frac{\alpha_{2}}{\alpha_{1}}=\frac{M}{N}\zeta^{l}
\]
for integers $M,N,l,m,n$. It probably makes most sense to observe $\thicksim_{3}$ because if
$\alpha\in{\varpi_{\epsilon/M,\zeta}}$ then $N\alpha\zeta^{k}\in{\varpi_{\epsilon,\zeta}}$
for any integers $k,\vert N\vert\leq M$.
It follows that we have no finiteness with respect to $\thicksim_{1}$ for any $\epsilon>0$ and $\zeta$ for which
$\lim_{n\to\infty} \Vert\alpha\zeta^{n}\Vert=0$ for some $\alpha\neq 0$, such as Pisot numbers $\zeta$.

\subsection{Fixed $\zeta$}  \label{eps}

We consider $\zeta>1$ fixed and interpret the results of
Section~\ref{fixedzeta} in terms of the variable $\epsilon>0$.
Subsequent to Corollary~\ref{kor2} we noticed that
$\varpi_{\epsilon,\zeta}\neq \emptyset$ implies $\varpi_{\epsilon,\zeta}\geq \vert \mathbb{Z}\vert$.
We now agree on $\sup{\{\emptyset\}}=0$. Recall the quantities $\tilde{\epsilon}_{i}(\zeta)$ from
Section~\ref{outline}.
The property $\lim_{n\to\infty} \Vert \alpha\zeta^{n}\Vert=0$ for some $\alpha\neq 0$
implies $\tilde{\epsilon}_{1}(\zeta)=0$, but not necessarily vice versa.
In particular, Theorem~\ref{pisot} implies $\tilde{\epsilon}_{1}(\zeta)=0$
for any Pisot number $\zeta$. Theorem~\ref{muskatottonel} further implies
$\tilde{\epsilon}_{1}(\zeta)\leq \frac{1}{2(\zeta-1)}$.
Concerning $\tilde{\epsilon}_{2}$, Theorem~\ref{muskatottonel} implies that for any $\zeta>1$
we have $\tilde{\epsilon}_{2}(\zeta)\leq \frac{1}{\zeta-1}$.
Proposition~\ref{lemur} suggests that we should expect $\tilde{\epsilon}_{2}(\zeta)\leq \frac{1}{2(\zeta-1)}$
for almost all $\zeta>1$ in the sense of Lebesgue measure. On the other hand, Theorem~\ref{machete}
implies $\tilde{\epsilon}_{2}(\zeta)\geq \frac{1}{2(\zeta+1)}$ for all $\zeta>1$. We sum up
some of these observations in a theorem which slightly extends Theorem~\ref{epsilont}.

\begin{theorem} \label{epsilon}
For any $\zeta>1$ we have
\[
0\leq \tilde{\epsilon}_{1}(\zeta)\leq \min\left\{\frac{1}{2},\frac{1}{2(\zeta-1)}\right\},
\qquad \frac{1}{2(\zeta+1)}\leq \tilde{\epsilon}_{2}(\zeta)\leq \min\left\{\frac{1}{2},\frac{1}{\zeta-1}\right\}.
\]
For any $\zeta\in{(1,\infty)\setminus \Gamma}$
we have $\tilde{\epsilon}_{2}(\zeta)\leq \min\{\frac{1}{2},\frac{1}{2(\zeta-1)}\}$.
\end{theorem}

Since $\lim_{\zeta\to\infty}\frac{1}{2(\zeta-1)}/\frac{1}{2(\zeta+1)}=1$,
assuming the existence of arbitrarily large $\zeta\notin{\Gamma}$, we infer
the lower bound for $\tilde{\epsilon}_{2}$ is optimal up to any factor greater $1$.
We will proof similar unconditioned results for the other bounds in Section~\ref{sektion3}.
By virtue of Remark~\ref{auchremark}, the $0$-symmetry of the intervals connected to
$\varpi_{\epsilon,\zeta}$ is only needed in the last claim,
which can also be extended by replacing $1/2$ by some other constant in the definition of $\Gamma$.
Thus for any $\zeta>1$ and given interval $I$  modulo $1$ of length greater than $1/(\zeta-1)$,
there exists $\alpha\neq 0$ such that $\{\alpha\zeta^{n}\}$ lies in $I$
for all large $n$. The results concerning $\tilde{\epsilon}_{2}$ allow a similar interpretation
with interval length $2/(\zeta-1)$.

\subsection{The special case of algebraic $\zeta>1$}  \label{sektion3}

In the case of algebraic numbers $\zeta>1$, some bounds in
Theorem~\ref{epsilon} can be refined with a result due to Dubickas.
Combination with Theorem~\ref{epsilon} will lead to the proof of Theorem~\ref{hot}.

For $\zeta$ a Pisot number, we know due to Theorem~\ref{pisot}
that $\cap_{\epsilon>0} \varpi_{\epsilon,\zeta}\neq \emptyset$
and hence in particular $\tilde{\epsilon}_{1}(\zeta)=0$.
Otherwise, if $\zeta>1$ is algebraic but not a Pisot number or a Salem number and $\alpha\neq 0$,
Dubickas~\cite[Theorem~1]{5} showed that
\begin{equation} \label{eq:dub}
\limsup_{n\to\infty} \Vert\alpha\zeta^{n}\Vert\geq \frac{1}{\min\{L(\zeta),2l(\zeta)\}}.
\end{equation}
The same holds for Pisot and Salem numbers and all $\alpha\notin{\mathbb{Q}(\zeta)}$.
More generally, the expression $1/\min\{L(\zeta),2l(\zeta)\}$ is a lower bound for the minimum
distance from the smallest to the largest limit point of $\{\alpha\zeta^{n}\}$.
Here $L(\zeta)$ is defined as in Section~\ref{cardgap} and
$l(\zeta)=l(P)$ is the infimum among all $L(PG)$ where $G\in{\mathbb{R}[X]}$ runs over all
polynomials with either leading or constant coefficient $1$, where $P\in{\mathbb{Z}[X]}$ is the
minimal polynomial of $\zeta$ in lowest terms. Combination of \eqref{eq:dub}
and Theorem~\ref{epsilon} yields for $\zeta>1$ algebraic not a Pisot or a Salem number
\begin{equation} \label{eq:tag}
\frac{1}{\min\{L(\zeta),2l(\zeta)\}}\leq \tilde{\epsilon}_{1}(\zeta) \leq
\min\left\{\frac{1}{2},\frac{1}{2(\zeta-1)}\right\}.
\end{equation}
In particular $\tilde{\epsilon}_{1}(\zeta)\neq 0$.
Furthermore, the estimate \eqref{eq:tag} yields the criterion stated in Theorem~\ref{hot}
for an algebraic number to be a Pisot or Salem number. To exclude the case that such $\zeta$ is a
Salem number and thus prove Theorem~\ref{hot}, it suffices to notice that Dobrowolski~\cite{dobro}
showed that any complex polynomial $P\in\mathbb{C}[X]$ with a root on the unit circle 
satisfies $L(P)\geq 2 M(P)$. Hence $\zeta \leq M(P)\leq L(P)/2<L(P)/2+1$ for any Salem number $\zeta$
with minimal polynomial $P$. 
We add a remark concerning \eqref{eq:tag} and Theorem~\ref{hot}.
\begin{remark} \label{mahler}
The estimate $\zeta-1>l(\zeta)$ in view of \eqref{eq:tag}
would allow the conclusion that $\zeta$ must be a Pisot or a Salem number, but it cannot be satisfied.
The estimate $M(\zeta)\leq l(\zeta)$ for all algebraic $\zeta$ and $M(\zeta)=M(P)$
the Mahler measure of the minimal polynomial $P$ of $\zeta$ defined in \eqref{eq:tiefet},
is known~\cite{dubbull}. This would lead to $\zeta-1>l(\zeta)\geq M(\zeta)\geq \zeta>\zeta-1$, contradiction.
\end{remark}

We now allow $\zeta$ to be a Pisot or a Salem number.
Since $\mathbb{Q}(\zeta)$ is countable, the estimate \eqref{eq:dub} for $\alpha\notin{\mathbb{Q}(\zeta)}$
and Theorem~\ref{epsilon} further imply
\begin{equation} \label{eq:nacht}
\frac{1}{\min\{L(\zeta),2l(\zeta)\}}\leq \tilde{\epsilon}_{2}(\zeta) \leq
\min\left\{\frac{1}{2},\frac{1}{\zeta-1}\right\}
\end{equation}
for any algebraic $\zeta>1$. The consequences $\zeta-1\leq L(\zeta)$ and $\zeta-1\leq 2l(\zeta)$
are already implied by \eqref{eq:hoehet} and Remark~\ref{mahler}, respectively.
Moreover, we deduce that the condition $2(\zeta-1)>L(\zeta)$
implies $\frac{1}{2(\zeta-1)}< \frac{1}{L(\zeta)}\leq \tilde{\epsilon}_{2}(\zeta)$ and hence that
$\zeta$ is exceptional in the sense of Definition~\ref{exceptional}.
Recall this condition is satisfied for the Pisot numbers
$\zeta_{m,b}$ from Section~\ref{outline} for large $b$ defined above and the
quotient $L(\zeta_{m,b})/(\zeta_{m,b}-1)$ tends to
$1$ as $b\to\infty$. The same applies to any integer $\zeta>2$.
Similarly to the polynomials $P_{m,b}$ defined in Section~\ref{outline},
consider polynomials of the form $Q_{m,b}(X)=2X^{m}-bX^{m-1}-1$.  The largest real root
$\eta_{m,b}$ of $Q_{m,b}(X)$ is larger $b/2$ and $L(Q_{m,b})=b+3$,
such that $L(\eta_{m,b})/(\eta_{m,b}-1)>2$ is arbitrarily close to $2$ if $b$ is large.
Since $Q_{m,b}(X)$ is no Pisot or Salem polynomial we may apply
\eqref{eq:tag}. Summing up, we infer Theorem~\ref{notintuitiver}. Its claim can be
interpreted in the way that the upper bounds for
$\tilde{\epsilon}_{1}, \tilde{\epsilon}_{2}$ in Theorem~\ref{epsilon}
(or equivalently those in \eqref{eq:dub}) are not far from being optimal. Moreover
Theorem~\ref{notintuitiver} implies that there exist exceptional Pisot numbers of any given degree.

Even though any Pisot number belongs to $\Gamma$, see Section~\ref{fixedzeta},
the claim concerning $\tilde{\epsilon}_{2}$ reinterpreted in terms of paths of the tree from Section~\ref{fixedzeta}
seems not too intuitive. Given an exceptional Pisot number, for {\em any} given start value $N_{0}\geq 1$,
most paths in the corresponding tree $\mathscr{T}=\mathscr{T}(\zeta,1/L(\zeta),N_{0})$
from Section~\ref{fixedzeta} with root $N_{0}$ must be deterministic,
i.e. in the path $N_{0},N_{1},\ldots$ the values $\{N_{j}\zeta\}$ avoid the symmetric neighborhood
$I(\zeta):=[1-\frac{\zeta-1}{L(\zeta)},\frac{\zeta-1}{L(\zeta)}]\neq \emptyset$
of $1/2$ for all large $j$. Clearly, $I(\zeta)$ is an arbitrarily large
subinterval of the entire interval $[0,1]$ if $L(\zeta)/(\zeta-1)$ is sufficiently close to $1$.
Moreover, each path leads to an element of $\mathbb{Q}(\zeta)$ via the correspondence
from Section~\ref{fixedzeta}, more precisely $\alpha=\lim_{j\to\infty} N_{j}/\zeta^{j}\in{\mathbb{Q}(\zeta)}$.
It is not obvious how to prove all of this in an elementary way without Dubickas' result.
Numerical computations for $\zeta=\zeta_{2,4}=2+\sqrt{5}$ the root of $X^{2}-4X-1$ and various values of $N_{0}$
affirm however that the fractional parts $\{N_{j}\zeta\}$ are near integers, in particular
avoid the corresponding interval
\[
I(\zeta_{2,4}):=\left[1-\frac{\zeta_{2,4}-1}{L(\zeta_{2,4})},\frac{\zeta_{2,4}-1}{L(\zeta_{2,4})}\right]
=\left[\frac{5-\sqrt{5}}{6},\frac{1+\sqrt{5}}{6}\right]\approx[0.4607,0.5393],
\]
for most paths and rather small $j$. The continued fraction expansion of many of the resulting elements
in $\mathbb{Q}(\zeta_{2,4})$ end in period $\overline{4}$. Observe $\zeta_{2,4}=[4;\overline{4}]$.
For $N_{0}\in{\{1,3\}}$ there is only one path given by $N_{j+2}=4N_{j+1}+N_{j}$ for $j\geq -1$
and suitable $N_{-1}$. For $N_{0}=2$ there are only two paths with $N_{1}=8$ and $N_{1}=9$
respectively, and $N_{j+2}=4N_{j+1}+N_{j}$ for $j\geq 0$. For $N_{0}=3$,
if we increase the avoided interval $I(\zeta_{2,4})$ to say $[0.15,0.85]$, which corresponds to a rise of
from $\epsilon=1/L(\zeta)=1/6$ to $\epsilon=0.85/(1+\sqrt{5})\approx 0.2627$,
it seems there is a non-deterministic path given by $N_{-1}=1$ and the
recurrence $N_{j+2}=4N_{j+1}+N_{j}-1$ for $j\geq -1$, and the resulting tree
$\mathscr{T}(\zeta_{2,4},0.2627,3)$ is isomorphic to $\mathscr{T}^{\prime}$ described
in Section~\ref{fixedzeta}. In particular the set of paths is no longer finite.
For any larger avoided interval or value of $\epsilon$ there should be uncountably many.
Recall also that $\alpha\in{\varpi_{\epsilon,\zeta}}$ does not necessarily induce
a path in $\mathscr{T}(\zeta,\epsilon,N_{0})$ for some $N_{0}$, only the reverse claim is proved.

On the other hand, the result concerning $\tilde{\epsilon}_{1}$ is intuitive.
For $\epsilon=\frac{\delta}{2(\zeta-1)}$ with $\delta\in{(0,1)}$, consider
the recursive process defined by $N_{j+1}=\scp{N_{j}\zeta}$ as long as $\zeta\cdot[N_{j}-\epsilon,N_{j}+\epsilon]$
contains the neighborhood $[N_{j+1}-\epsilon,N_{j+1}+\epsilon]$ of $N_{j+1}$,
following the proof of Theorem~\ref{muskatottonel}. If for some
start value $N_{0}$ the process does not stop,
which means $(\{N_{j}\zeta\})_{j\geq 0}$ avoids some interval modulo $1$ centered at $1/2$,
it leads to $\alpha\in{\varpi_{\epsilon,\zeta}}$. The interval length tends to $0$ as $\delta\to 1$.
If otherwise for any start value $N_{0}$ the process stops at some index $j=j(N_{0})$,
although the process only yields a sufficient criterion,
we should expect that there is no arising $\alpha\in{\varpi_{\epsilon,\zeta}}$.
We should also expect that $\{N_{j}\zeta\}$ is dense
in $[0,1]$ for any start value $N_{0}$ and any algebraic $\zeta>1$ which is no Pisot number.
This argument suggests to conjecture that almost all real $\zeta>2$ satisfy
$\tilde{\epsilon}_{1}=\frac{1}{2(\zeta-1)}$ too. Further we add
that there is no reason why any Salem number $\zeta$ should belong to $\Gamma$. Thus it
is reasonable to expect that no Salem number is exceptional and hence
only Pisot numbers can satisfy $2(\zeta-1)>L(\zeta)$.

\subsection{The case of rational $\zeta>1$}  \label{sektions}
For the remainder of the paper we restrict to the case of rational $\zeta>1$.
We start with general comments on the distribution of powers of rationals modulo 1.
It has been intensely studied, but is still poorly understood.
For instance, it is unknown if the sequence $\{(3/2)^{n}\}$ is dense modulo $1$.
We quote some known results.
From Theorem~\ref{pisot} we infer that $\Vert \alpha\zeta^{n}\Vert$ does not
converge to $0$ as $n\to\infty$ for rational $\zeta>1$
which is no integer and any $\alpha\neq 0$.
This is equivalent to $\bigcap_{\epsilon>0} \varpi_{\epsilon,\zeta}=\emptyset$
for $\zeta\in{\mathbb{Q}\setminus{\mathbb{Z}}}$.
More generally, Vijayavagharan~\cite{viya} (see also~\cite{viya2}) proved
that the set of accumulation points of $(p/q)^{n}\bmod 1$
is always infinite unless $p/q$ is an integer. Pisot~\cite{pisot2}
generalized this by showing that in fact $\alpha \zeta^{n}\bmod 1$ has infinitely many limit
points if $\alpha\neq 0$ is real and $\zeta>1$ algebraic, unless in the case where $\zeta$ is
a Pisot number and $\alpha\in{\mathbb{Q}(\zeta)}$ where it must fail by Theorem~\ref{pisot}.
Dubickas~\cite{dubi} gave another proof of this fact.

Now we put our focus predominately on the values $\tilde{\epsilon}_{1}, \tilde{\epsilon}_{2}$.
We point out that in contrast to prior results, in the present
section the symmetry of the intervals with respect to $0$ modulo $1$ is mostly important.
It turns out that it is useful to distinguish the cases of $\zeta$ an integer or not.
First let $\zeta>1$ be an integer. Then any rational number of the form
$\alpha=M\zeta^{b}$ for $M,b$ integers leads to integers $\alpha\zeta^{n}$
for any $n\geq \vert b\vert$.
Hence $\tilde{\epsilon}_{1}(\zeta)=0$ for $\zeta>1$ an integer. Conversely,
writing $\alpha$ in base $\zeta$, it is not hard to see that
$\lim_{n\to\infty} \Vert \alpha \zeta^{n}\Vert=0$ implies
$\alpha\zeta^{n}\in{\mathbb{Z}}$ for all large $n$, and to deduce that
$\alpha$ must be of the given form.

For rational $\zeta=p/q>1$, the lower bound in \eqref{eq:tag} can be shown to be $1/L(\zeta)=1/(p+q)$.
Recall the notion of $\tau(p/q)$ from Section~\ref{outline}.
Dubickas improved his result \eqref{eq:dub} from Section~\ref{sektion3} for $\zeta\in{\mathbb{Q}}$ by showing
that for every rational $\zeta=p/q>1$ and $\alpha\neq 0$, with $\alpha$
irrational if $\zeta$ is an integer, the estimate
\begin{equation} \label{eq:dubgl}
\limsup_{n\to\infty}\Vert \alpha(p/q)^{n}\Vert \geq \tau(p/q)=
\frac{1}{2q}\left(1-\left(1-\frac{q}{p}\right) \prod_{m\geq 0} \left(1-\left(\frac{q}{p}\right)^{2^{m}}\right) \right)
>\frac{1}{p+q}
\end{equation}
holds~\cite[Theorem~3]{5}. We combine the facts from the integer and the non-integer case.

\begin{proposition} \label{weberknecht}
For rational $\zeta>1$ we have $\varpi_{\epsilon,\zeta}\neq \emptyset$ for every $\epsilon>0$
if and only if $\zeta$ is an integer and in this case
$\bigcap_{\epsilon>0}\varpi_{\epsilon,\zeta}=\mathscr{R}(\zeta):=\{M\zeta^{b}: M\in{\mathbb{Z}\setminus\{0\}},
b\in{\mathbb{Z}}\}$.
\end{proposition}

In view of \eqref{eq:dubgl}, for rational $\zeta=p/q>1$ which is not an integer we have
\[
\tilde{\epsilon}_{2}(p/q)\geq\tilde{\epsilon}_{1}(p/q)\geq \tau(p/q).
\]
Similarly for $\zeta=p/1>1$ an integer, since the numbers that violate
\eqref{eq:dubgl}, including $\mathscr{R}(\zeta)$, are rational and thus their set is countable, we have
\begin{equation} \label{eq:faktum}
\tilde{\epsilon}_{2}(p/1)\geq \tau(p/1).
\end{equation}
As mentioned in~\cite{dubidu}, it can be shown that
$\tau(p/q)>\frac{1}{p}-\frac{q^{2}}{p^{3}}$ for any rational $p/q>1$.
Since $\frac{1}{2(p-1)}<\frac{1}{p}-\frac{1}{p^{3}}<\tau(p/1)$ for $p\geq 3$, this confirms the claim from
Section~\ref{fixedzeta} that the set $\mathbb{N}_{\geq 3}$ is contained in the exceptional set defined there.
For $\zeta=p/1>1$ an integer, \cite[Corollary~2]{5} shows that for the choice $\alpha=\tau(p/1)$
there is actually equality in \eqref{eq:dubgl}. As mentioned subsequent to Corollary~\ref{kor2},
this means $\varpi_{\epsilon,\zeta}$ is at least countable for $\epsilon=\tau(p/1)$,
since it contains the number $\tau(p/1)p^{m}$ for any integer $m\geq 0$.
It is however not clear from the construction in~\cite{5}
if there are uncountably many $\alpha\in{\varpi_{\epsilon,\zeta}}$
for given $\epsilon>\tau(p/1)$, which together with \eqref{eq:faktum}
would imply $\tilde{\epsilon}_{2}(p/1)=\tau(p/1)$.
Theorem~\ref{muskatottonel} gives the weaker
upper bound $1/(\zeta-1)=1/(p-1)$ for $\tilde{\epsilon}_{2}(p/1)$.
We can improve this bound with an explicit construction.
Consider the set $Z\subseteq \mathbb{R}$ that consists of $\alpha\in{(0,1)}$ whose base $p$ expansion
has the following properties: only the digits $0,p-1$ appear,
there are at most two consecutive $p-1$ digits and the distance between blocks with two consecutive digits $p-1$
tends to infinity, and the digits $0$ are isolated. In other words, it is derived from the periodic digit
sequence $\overline{0,p-1}$ by plugging in single additional $p-1$ digits at large distances.
The set $Z$ is obviously uncountable. Furthermore, distinguishing the cases of
$n$ such that the first digit after the comma is $0$ and $p-1$ respectively,
leads to
\begin{align*}
\{\alpha\zeta^{n}\}&\leq \frac{p-1}{p^{2}}+\frac{p-1}{p^{3}}+\frac{p-1}{p^{5}}
+\frac{p-1}{p^{7}}+\cdots+\frac{p-1}{p^{2l+1}}+\frac{p-1}{p^{2l+2}}+\frac{p-1}{p^{2l+4}}+\ldots,  \\
\frac{p-1}{p}&\leq \{\alpha\zeta^{n}\}\leq \frac{p-1}{p}+\frac{p-1}{p^{3}}+\frac{p-1}{p^{5}}+\cdots,
\end{align*}
respectively for $\alpha\in{Z}$.
By construction we may let $l\to\infty$ as $n\to\infty$, so evaluating the geometric series
leads to the bounds $(p^{2}+p-1)/(p^{3}+p^{2})$ and $1/(p+1)$ for
$\Vert \alpha p^{n}\Vert$, respectively. The first bound is larger, thus
$\limsup_{n\to\infty} \Vert \alpha p^{n}\Vert\leq 1/p-1/(p^{3}+p^{2})$ for all $\alpha\in{Z}$.
Summarizing the facts on the case $\zeta=p/1$ proves Theorem~\ref{jep}.

Now we treat the case $\zeta\in{\mathbb{Q}\setminus \mathbb{Z}}$. In this case we can refine the recursive methods
from Theorem~\ref{muskatottonel} and Theorem~\ref{machete}. First
recall the definitions and remarks subsequent to Theorem~\ref{muskatottonel}.
For $\zeta\in{\mathbb{Q}\setminus \mathbb{Z}}$ the numbers
$\{N_{j}\zeta\}$ in any path are contained in the finite set $\{0,1/q,\ldots,(q-1)/q\}$.
Thus if $q$ is odd then certainly no path will be dense at $1/2$ and so $\zeta\in{\Gamma}$.
For even $q$, in the generic case we should expect $\{N_{j}\zeta\}=1/2$
infinitely often in any path, so $\zeta\notin{\Gamma}$ and thus $\zeta$ is not exceptional.
It is hard to predict if this heuristic argument applies to all such rationals.
However, we can slightly improve the bound $1/(\zeta-1)=\frac{q}{p-q}$ from
Theorem~\ref{muskatottonel} for all rational $\zeta>1$.
This will in particular imply that all rationals $p/2$ for $p$ odd are not exceptional.

\begin{proposition}  \label{p}
Let $\zeta=p/q$ with $p>q\geq 2$ and $(p,q)=1$. Then for any $\epsilon\geq \frac{q-1}{p-q}$ the set
$\varpi_{\epsilon,\zeta}$ is uncountable.
\end{proposition}

\begin{proof}
First observe that for every $N_{0}$, the sequence $(N_{j})_{j\geq 1}$ defined by $N_{j+1}=\scp{\zeta N_{j}}$
cannot have the property $\{\zeta N_{j}\}=0$ for all $j\geq j_{0}$.
Without loss of generality assume $j_{0}=0$. Indeed, if $\nu_{q}(N_{0})$
denotes the largest power of $q$ dividing $N_{0}$, then $\zeta\cdot N_{\nu_{q}(N_{0})}$ is not an integer.
Hence $\Vert\zeta N_{j}\Vert\geq 1/q$ for some $j\geq 0$.
It suffices to require $\epsilon\zeta\geq 1+\epsilon-1/q$ to ensure that
for any such index $j$ the corresponding interval $\zeta\cdot[N_{j}-\epsilon,N_{j}+\epsilon]$ of
length $2\epsilon\zeta$ and midpoint $\zeta N_{j}$
contains two consecutive integers. The condition is equivalent
to $\epsilon\geq \frac{q-1}{p-q}$, and repeating this argument shows that the set of paths
and thus $\varpi_{\epsilon,\zeta}$ is uncountable, as carried out preceding Proposition~\ref{lemur}.
\end{proof}

For odd $q$, we can also slightly improve the upper bound for $\tilde{\epsilon}_{1}(p/q)$ from Theorem~\ref{epsilon}.

\begin{proposition} \label{q}
Let $\zeta=p/q$ with $p>q\geq 2$ and $(p,q)=1$ and $q$ odd. Then for any
$\epsilon\geq \frac{q-1}{2q}\cdot\frac{1}{\zeta-1}=\frac{q-1}{2(p-q)}$
the set $\varpi_{\epsilon,\zeta}$ is dense in $\mathbb{R}$.
\end{proposition}

\begin{proof}
Again we follow the proof of Theorem~\ref{muskatottonel}. We have to show that
for $\epsilon$ as in the proposition,
in any step the interval $\zeta\cdot[N_{j}-\epsilon,N_{j}+\epsilon]$ of length
$\zeta\epsilon$ contains the symmetric neighborhood of $0$ of length $2\epsilon$ of
some integer. Since $q$ is odd,
the fractional part $\{N_{j}\zeta\}$ has distance at least $1/(2q)$ from $1/2$. Thus it
suffices to have $\zeta\epsilon\geq 1/2+\epsilon-1/(2q)$, which leads to the given bound,
to guarantee the claim.
\end{proof}

For $q=2$, Dubickas~\cite{dubidu} showed that
\begin{equation} \label{eq:dualte}
 \left\Vert \alpha \left(\frac{p}{2}\right)^{n}\right\Vert \leq \frac{1}{p}, \qquad n\geq 0,
\end{equation}
has a solution $\alpha\neq 0$ for any fixed odd $p\geq 3$. As remarked in~\cite{dub},
it follows from \eqref{eq:dubgl} that the bound in \eqref{eq:dualte} cannot be improved
to $p^{-1}-4p^{-3}<\tau(p/1)$. Finally, the bound from Theorem~\ref{machete} can be slightly improved
for $q$ odd with a similar method.

\begin{proposition} \label{r}
Let with $p>q\geq 2$ and $(p,q)=1$ and $q$ odd. Then for any
$\epsilon< \frac{q+1}{2q}\cdot\frac{1}{\zeta+1}=\frac{q+1}{2(p+q)}$
the set $\varpi_{\epsilon,\zeta}$ is at most countable.
\end{proposition}

\begin{proof}
Proceed as in the proof of Theorem~\ref{machete}. Note that since $q$ is odd
we have $\vert \{\zeta M\}-1/2\vert\geq 1/2q$ for any integer $M$.
Hence, given $M_{n}$, for $\vert M_{n}\zeta-M_{n+1}\vert\leq \epsilon$ to determine a unique $M_{n+1}$,
it suffices to assume $(\zeta+1)\epsilon<1/2+1/(2q)$. Rearrangement leads to the given bound.
\end{proof}

Now we have all ingredients to prove Theorem~\ref{gutkort}.

\begin{proof}[Proof of Theorem 2.6]
Combination of \eqref{eq:dubgl}, \eqref{eq:dualte}, Theorem~\ref{epsilon}, Proposition~\ref{p}, Proposition~\ref{q}
and Proposition~\ref{r} in terms of the quantities $\tilde{\epsilon}_{1},\tilde{\epsilon}_{2}$.
\end{proof}

We enclose several remarks concerning Theorem~\ref{gutkort}.
The lower bounds are always non-trivial, whereas the upper bounds
are only in case of $\zeta$ not too small.
Moreover, for $q\geq 2$, indeed $\tau(p/q)< \frac{q}{2(p-q)}$ which enables the first inequality.
Recall that for $q=1$, we have $\frac{1}{p+1}<\tau(p/1)$ such that $\tau(p/1)\leq \frac{q}{2(p-q)}$
cannot hold for any $p\geq 3$. However, $q=1$ is excluded in Theorem~\ref{gutkort}.
It further follows from $1/(p+q)<\tau(p/q)$ that for $q=2$ the refined upper bound
$\frac{q-1}{2(p-q)}$ for $\tilde{\epsilon}_{1}(p/q)$ is not valid at least for $p\geq 7$.
This corresponds to the fact that the fractional parts $\{N_{j}(p/2)\}$ must equal $1/2$ infinitely
often in any path in Proposition~\ref{q} by a very similar argument as in the proof of Proposition~\ref{p}.
It is easily checked that the maximum in the lower bound for $\tilde{\epsilon}_{2}$ coincides with
$\frac{q}{2(p+q)}$  (resp. $\frac{q+1}{2(p+q)}$) unless $q=2$. In particular one may drop
the expression $\tau(p/q)$ in the maximum in \eqref{eq:diebt} without any loss.
Notice also that the remarks preceding Proposition~\ref{p}
suggest that actually $\frac{q}{2(p-q)}$ should be an upper bound for $\tilde{\epsilon}_{2}(p/q)$
for even $q\geq 4$ too (this is true for $q=2$ since the bound coincides with $\frac{q-1}{p-q}$).

Next we prove Theorem~\ref{hundertt}, which confirms the bound $1/(p+q)$ from \eqref{eq:tag}
for rational $\zeta=p/q>1$ with an easier proof and contains some additional new information.
The proof is related to the proof of Proposition~\ref{p}.

\begin{proof} [Proof of Theorem 2.7]
First we show it is true for any $\epsilon$ with strict inequality $\epsilon<1/(p+q)$.
Assume the claim is false. Then in particular
$\Vert \alpha\zeta^{n}\Vert < 1/(p+q)$ for all $n\geq n_{0}(\alpha,\zeta)$.
Write $\alpha\zeta^{n}=A_{n}+\delta_{n}$ with integers $A_{n}=\scp{\alpha\zeta^{n}}$ and
$-1/(p+q)< \delta_{n}< 1/(p+q)$. Then $\alpha\zeta^{n+1}=\frac{p}{q}A_{n}+\frac{p}{q}\delta_{n}$.
If $\frac{p}{q}A_{n}$ is no integer, then it has distance at least $1/q$ to the nearest
integer. But $\vert \frac{p}{q}\delta_{n}\vert< \frac{p}{q}\cdot\frac{1}{p+q}=\frac{p}{q(p+q)}$.
So we have
\[
\Vert \alpha\zeta^{n+1}\Vert > \frac{1}{q}-\frac{p}{q(p+q)}=\frac{1}{p+q}>\delta_{n+1}
\]
by triangular inequality, a contradiction. Hence $\frac{p}{q}A_{n}$ must be an integer and clearly
$\scp{\alpha \zeta^{n+1}}= \frac{p}{q}A_{n}=A_{n+1}$ again by $\vert \frac{p}{q}\delta_{n}\vert< \frac{p}{q(p+q)}$.
However, this applies to $n+1,n+2,\ldots$ as well by the same argument.
Hence $A_{n+j}=(p/q)^{j}A_{n}$ for all $0\leq j\leq l$. Since $\alpha\neq 0$ by definition,
and we may assume that $n$ is large enough such that $A_{n}\neq 0$,
the integer $A_{n}\neq 0$ can only be divisible by at most $\log A_{n}/\log q$ powers of $q$.

Note that $A_{n}=\vert\scp{\alpha(p/q)^{n}}\vert\leq \vert\alpha\vert(p/q)^{n}+1/2$. Thus
\[
l\leq \frac{\log A_{n}}{\log q} \leq n\cdot\left(\frac{\log p}{\log q}-1\right)+\log \vert\alpha\vert+o(1),
\qquad n\to\infty.
\]

It remains to extend the result to $\epsilon=1/(p+q)$. If there are at most finitely many integers
$m$ such that $\Vert \alpha(p/q)^{m}\Vert=1/(p+q)$, then the assertion is
implied by our proof of the case $\epsilon<1/(p+q)$. We show this is always true.
For any $m$ with equality, we have the equation $\alpha(\frac{p}{q})^{m}=M_{m}\pm \frac{1}{p+q}$
for an integer $M_{m}$. It follows $\alpha$ must be rational too, say $\alpha=a/b$ with integers $a,b$,
and the equation becomes $(p+q)(ap^{m}-M_{m}bq^{m})=\pm bq^{m}$. For a prime $r$ denote
by $\nu_{r}(.)$ the multiplicity of $r$.
By $(p,q)=1$, any prime divisor $r$ of $q$ is not contained in $p+q$, and for $m>\nu_{r}(a)$ we
have $\nu_{r}(ap^{m}-M_{m}bq^{m})= \nu_{r}(a)$. On the other hand, $\nu_{r}(bq^{m})\geq \nu_{r}(q^{m})\geq m$.
Hence for any $m>\nu_{r}(a)$ we cannot have equality.
\end{proof}

\begin{remark}
It suffices to take $n\geq n_{0}=n_{0}(\alpha,\zeta):=\max\{0,-\log \vert\alpha\vert/\log (p/q)\}$
to ensure $A_{n}\neq 0$. Theorem~\ref{hundertt} in particular yields $l\ll_{\alpha,\zeta} n$.
\end{remark}

\begin{remark}
The last part of the proof could have been inferred from the more general~\cite[Lemma~2.1]{dubidu}.
It asserts that for $p/q\in{\mathbb{Q}\setminus \mathbb{Z}}$ the equation
$\{\alpha(p/q)^{n}\}=t$ can have only finitely many solutions $n$ for any $t\in{[0,1)}$ and
fixed $\alpha\neq 0$. In this context, we want to add that if $\{\alpha\zeta^{n}\}=t$ for real
$\zeta\neq 0, \alpha\neq 0, t\in{[0,1)}$ and at least three values $n$, then $\alpha,\zeta,t$ have
to be all algebraic. Indeed, if there exist integers $n_{i},N_{i}$ such that
$\alpha\zeta^{n_{i}}=N_{i}+t$ for $1\leq i\leq 3$, then
$(\zeta^{n_{3}}-\zeta^{n_{2}})/(\zeta^{n_{2}}-\zeta^{n_{1}})\in{\mathbb{Q}}$.
This can be transformed in a polynomial equation with rational coefficients, so $\zeta$
must be algebraic. Thus $\alpha=(N_{2}-N_{1})/(\zeta^{n_{2}}-\zeta^{n_{1}})$ implies $\alpha$
must be algebraic as well, hence $t$ as well. On the other hand, for $\zeta$ a root of an integer,
$\alpha\in{\mathbb{Z}}$ and $t=0$, there are infinitely many integers $n$ such that $\{\alpha\zeta^{n}\}=t$.
It can be shown that at least for $t=0$, the restrictions $\zeta=\sqrt[m]{N}$ and $\alpha=\frac{A}{B}\zeta^{g}$
for integers $N,A,B,g$ are necessary too, see~\cite[Proposition~2.27]{schlei}.
\end{remark}

\subsection{The asymmetric case} \label{ende}
For sake of completeness we quote some more facts concerning the distribution
of $\alpha\zeta^{n}$ for rational $\zeta>1$ concerning
intervals mod $1$ whose center is not $0$. Many of these can be found (without proofs)
on the first page of~\cite{28} too. Tijdeman~\cite{22} showed that
\begin{equation} \label{eq:butter}
 0\leq \left\{\alpha \left(\frac{p}{q}\right)^{n}\right\}\leq \frac{q-1}{p-q}, \qquad n\geq 0
\end{equation}
has a solution $\alpha\in{[m,m+1)}$ for any rational number $p/q$ and $m\geq 1$.
We recognize the upper bound as the bound for $\tilde{\epsilon}_{2}$ in Theorem~\ref{gutkort},
where the interval has twice the length. The length for the $0$-symmetric interval
concerning $\tilde{\epsilon}_{1}$ in Theorem~\ref{gutkort} has the same
length for odd $q$ and is slightly larger for even $q$.
Clearly \eqref{eq:butter} never admits an improvement of the upper bound for $\tilde{\epsilon}_{1}$
in Theorem~\ref{gutkort}. In particular for $q=2$ and odd $p\geq 3$, we obtain from \eqref{eq:butter} that
\begin{equation} \label{eq:birne}
 0\leq \left\{\alpha \left(\frac{p}{2}\right)^{n}\right\}\leq \frac{1}{p-2}, \qquad n\geq 0
\end{equation}
has a solution $\alpha\in{[m,m+1)}$ for integer $m\geq 1$. Compare \eqref{eq:birne} to \eqref{eq:dualte}.

Dubickas bound \eqref{eq:tag}, or equivalently Theorem~\ref{hundertt}, show that the upper bounds
in \eqref{eq:butter} cannot be improved to $1/(p+q)$. In particular, the bounds in
\eqref{eq:dualte}, \eqref{eq:birne} cannot be replaced by
$1/(p+2)$ for any pair $(p,\alpha)$ with odd $p\geq 5$ and real $\alpha\neq 0$.
Conversely, the uniform bounds in \eqref{eq:tag} and Theorem~\ref{hundertt}
are not far from being optimal, in particular if $\zeta=p/q$ is large.

In the famous special case $\zeta=3/2$, it was shown in~\cite{24} that
\[
 \limsup_{n\to\infty} \left\{\alpha \left(\frac{3}{2}\right)^{n}\right\}-
 \liminf_{n\to\infty} \left\{\alpha \left(\frac{3}{2}\right)^{n}\right\}\geq \frac{1}{3}
\]
for any $\alpha>0$. More generally, Theorem~1 in~\cite{dubi2} due to Dubickas asserts
\[
\limsup_{n\to\infty} \left\{ \alpha\left(\frac{p}{q}\right)^{n}\right\} -
\liminf_{n\to\infty} \left\{\alpha\left(\frac{p}{q}\right)^{n}\right\} \geq \frac{1}{p}
\]
for $p/q\notin{\mathbb{Z}}$ greater than $1$ and all $\alpha\neq 0$,
such as all irrational $\alpha$ if $p/q=p$ is an integer. As pointed out in~\cite{dubi2},
in the integer case $\zeta=p/1$ the bound $1/p$ is sharp, and $\alpha$ with equality can be readily constructed.
For further results on (unions of) subintervals of $[0,1]$ containing the numbers $\{\alpha(p/q)^{n}\}$
for all $n\geq 1$ and a given rational $p/q$, or for which such $\alpha$ does not exist, see~\cite{dubidu}.
Choquet~\cite{23} proved there exists $\alpha$ such that $1/19\leq \{\alpha(3/2)^{n}\}\leq 1-1/19$ for
all $n\geq 1$. On the other hand, we have
\[
\inf_{\alpha\neq 0}\limsup_{n\to\infty} \left\Vert \alpha(p/q)^{n}-\frac{1}{2}\right\Vert \geq \frac{1-e(q/p)T(q/p)}{2q}
\]
due to Dubickas~\cite{5}, where $e(q/p)=1-q/p$ if $p+q$ is even and $e(q/p)=1$ if $p+q$ is odd
and $T(z):=\prod_{m\geq 0} (1-z^{2^{m}})$. Hence $1/19\approx 0.0526$
cannot be replaced by $1/2-(1-T(2/3))/4\approx 0.2856$.

\vspace{1cm} 

The author thanks the anonymous referee for the careful review and A. Dubickas for pointing out
the paper of E. Dobrowolski to improve the initial version of Theorem~\ref{hot}.

\end{document}